\documentclass[10pt,reqno]{amsart}
\usepackage{amsmath} 
\usepackage{amssymb}
\usepackage{dsfont}
\usepackage[dvips,draft,final]{graphics}
\usepackage[T1]{fontenc}
\usepackage{lmodern}
\usepackage{fancyhdr}
\usepackage{url}
\usepackage{enumerate}
\usepackage{color}
\usepackage{graphicx}
\usepackage{mathrsfs}

\newtheorem{theorem}{Theorem}[section]
\newtheorem{proposition}[theorem]{Proposition}
\newtheorem{lemma}[theorem]{Lemma}
\newtheorem{co}[theorem]{Corollary}

\theoremstyle{definition}
\newtheorem{rem}[theorem]{Remark}

\newcommand{\bel}{\begin{equation} \label}
\newcommand{\ee}{\end{equation}}

\newcommand{\pd}{\partial}

\newcommand{\R}{{\mathbb R}}

\newcommand{\re}{\mathfrak R}

\def\epsilon{\varepsilon}
\def\phi {\varphi}

\def\beq{\begin{equation}}
\def\eeq{\end{equation}}
\renewcommand{\leq}{\leqslant}
\renewcommand{\geq}{\geqslant}
\newcommand{\bea}{\begin{eqnarray}}
\newcommand{\eea}{\end{eqnarray}}
\newcommand{\beas}{\begin{eqnarray*}}
\newcommand{\eeas}{\end{eqnarray*}}

{

\providecommand{\abs}[1]{\left\lvert#1\right\rvert}
\providecommand{\norm}[1]{\left\lVert#1\right\rVert}

\title[Determination of non-compactly supported  electromagnetic potentials]{Determination of  non-compactly supported  electromagnetic potentials in unbounded closed waveguide}
\author{Yavar Kian}
\address{Aix Marseille Univ, Universit\'e de Toulon, CNRS, CPT, Marseille, France.}
\email{yavar.kian@univ-amu.fr}

\begin{document}
\begin{abstract}
We study the inverse problem  of determining a magnetic Schr\"odinger operator in an unbounded closed waveguide from boundary measurements. We consider this problem with a general  closed waveguide in the sense that we only require  our unbounded domain to be contained into an infinite cylinder. In this context we prove the unique recovery of the magnetic field and the electric potential associated with general bounded and non-compactly supported electromagnetic potentials. By assuming that the electromagnetic potentials are known on the neighborhood of the boundary outside a compact set, we  even prove the unique determination of the magnetic field and the electric potential from measurements restricted to a bounded subset of the infinite boundary. Finally, in the case of a waveguide  taking the form of an infinite cylindrical domain, we  prove the recovery  of the magnetic field and the electric potential from partial data corresponding to restriction of Neumann boundary measurements  to slightly more than half of the boundary. We establish all these results by mean of a new class of  complex geometric optics solutions and of Carleman estimates suitably designed for our problem stated in an unbounded domain and  with bounded electromagnetic potentials.

\medskip
\noindent
{\bf Keywords :} Inverse problems, elliptic equations, electromagnetic potential, Carleman estimate, unbounded domain, closed  waveguide,  partial data.

\medskip
\noindent
{\bf Mathematics subject classification 2010 :} 35R30, 35J15.
\end{abstract}
\maketitle


\section{Introduction}
\label{sec-intro}
\setcounter{equation}{0}
\subsection{Statement of the problem}
Let  $\Omega$ be an unbounded open   set of $\R^3$  corresponding  to a closed waveguide. Here by closed waveguide we mean that
there exists $\omega$  a $\mathcal C^2$  bounded open simply connected set of $\mathbb{R}^2$ such that the following condition is fulfilled
\bel{closed}\Omega \subset  \omega\times \R.\ee
 For $A\in L^\infty(\Omega)^3$, we define the magnetic Laplacian $\Delta_A$ given by
$$\Delta_A=\Delta+2iA\cdot\nabla+i\textrm{div}(A)-|A|^2.$$
According to  \cite[Theorem 3.4 page 223]{EE}, for any $u\in H^1(\Omega)$ and $\phi\in\mathcal C^\infty_0(\Omega)$, we have $u\phi\in W^{1,1}_0(\Omega)$, where $W^{1,1}_0(\Omega)$ denotes the closure of $\mathcal C^\infty_0(\Omega)$ in $W^{1,1}(\Omega)$. Therefore, using a density argument we can prove that, for any $u\in H^1(\Omega)$ and $A\in L^\infty(\Omega)^3$, we have div$(A)u\in D'(\Omega)$ and  $\Delta_A u\in D'(\Omega)$.
Thus, for $q \in L^\infty(\Omega;\mathbb C)$ and $u\in H^1(\Omega)$,  we can introduce the  equation
\bel{eq1}
\Delta_Au + q u  = 0,\quad\mbox{in}\ \Omega
\ee
in the sense of distributions.  Since we make no assumption on the boundary of $\Omega$, in a similar way to \cite{KU}, we define the trace map $\tau$ on $H^1(\Omega)$ by $\tau u=[u]$ with $[u]$  the class of $u$ in the quotient space $\frac{H^1(\Omega)}{H^1_0(\Omega)}$, where $H^1_0(\Omega)$ denotes the closure of $\mathcal C^\infty_0(\Omega)$ in $H^1(\Omega)$. We associate to any solution $u\in H^1(\Omega)$ of \eqref{eq1} the trace $ N_{A,q}u\in \left(\frac{H^1(\Omega)}{H^1_0(\Omega)}\right)'$, with $\left(\frac{H^1(\Omega)}{H^1_0(\Omega)}\right)'$ the dual space of $\frac{H^1(\Omega)}{H^1_0(\Omega)}$, defined by
$$\left\langle N_{A,q}u,\tau g\right\rangle_{\left(\frac{H^1(\Omega)}{H^1_0(\Omega)}\right)',\frac{H^1(\Omega)}{H^1_0(\Omega)} }:=-\int_\Omega (\nabla+iA)u\cdot \overline{(\nabla+iA)g}dx+\int_\Omega qu\overline{g}dx,\ g\in H^1(\Omega).$$
Here, by using a   density argument, one can prove that this map is well defined for $u$ solving \eqref{eq1} since for $g\in H^1_0(\Omega)$ the right hand side of this identity is equal to $0$.

Recall that for  $\Omega=\omega\times \R$ one can identify $\frac{H^1(\Omega)}{H^1_0(\Omega)}$ to $H^{\frac{1}{2}}(\pd\omega\times\R):=L^2(\R;H^{\frac{1}{2}}(\pd\omega)\cap H^{\frac{1}{2}}(\R;L^2(\omega))$. Then, for $u\in H^1(\Omega)$ solving \eqref{eq1} and $A\in W^{1,\infty}(\Omega)^3$, we have $\tau u=u_{|\pd\Omega}$ and $$N_{A,q}u=-\pd_{\nu_A}u=-\pd_\nu u-i(A\cdot \nu)u\in H^{-\frac{1}{2}}(\pd\omega\times\R)=(H^{\frac{1}{2}}(\pd\omega\times\R))',$$ with $\nu$ the outward unit normal vector to $\pd\omega\times \R$. This means that $-N_{A,q}$ is the natural extension of the magnetic normal derivative in non smooth setting for general unbounded domains satisfying \eqref{closed}. 

We introduce then the data
\bel{data} \mathcal D_{A,q}:=\{(\tau u, N_{A,q}u):\ u\in H^1(\Omega), \textrm{ $u$ solves \eqref{eq1}}\}.\ee
Note that for  $\Omega=\omega\times \R$, $A\in W^{1,\infty}(\Omega)^3$ and assuming that $0$ is not in the spectrum of $\Delta_A + q$ with Dirichlet boundary condition, $\mathcal D_{A,q}$ corresponds, up to the sign, to the graph of the so called Dirichlet-to-Neumann map associated with \eqref{eq1}. In this paper we consider the simultaneous recovery of the magnetic field associated with $A$ and $q$ from the data $\mathcal D_{A,q}$. We consider both results with full and partial data.

\subsection{Physical motivations}
Let us first observe that, the problem addressed in this paper is linked to the so called  electrical impedance tomography (EIT in short) method  and its applications in medical imaging  and geophysical prospection (see \cite{Uh} for more detail). The statement of the present inverse problem in an unbounded closed waveguide can be addressed in the context of problems of transmission to long distance or transmission through particular structures, with important ratio length-to-diameter, such as nanostructures. Here the goal of  the inverse  problem can be described as the unique recovery of an  electromagnetic impurity  perturbing the guided propagation (see  \cite{CL,KBF}). Let us also mention that in this paper we consider general closed waveguides, only subjected to condition \eqref{closed}, that have not necessary a cylindrical shape comparing to other related works like \cite{CKS2,CKS3,Ki4}. This means that we can consider our inverse problem in closed waveguides with different types of geometrical
deformations, including bends and twisting, which can be used in several context for improving the propagation of signals (see for instance \cite{Sr}). 
 
\subsection{State of the art}
We recall that the Calder\'on problem, addressed first in  \cite{Ca}, has attracted many attention over the last decades (see for instance \cite{Ch,Uh} for an overview of several aspects of this problem). The first positive answer to this problem  in dimension $n \geq 3$ has been addressed by Sylvester and Uhlmann in \cite{SU}. Here the authors introduced the so called complex geometric optics (CGO in short) solutions which remain one of the most important tools for the study of this problem.
This last result has been extended in several  way. For instance, we can mention the problem stated with partial data by \cite{BU}   and improved by   \cite{KSU}. One of the first results about the recovery, modulo gauge invariance, of electromagnetic potentials has been addressed in \cite{Suu} where the author proved the determination of magnetic field associated with magnetic potentials $A$ lying in $ W^{2,\infty}$ by assuming that the magnetic field is sufficiently small. The smallness assumption of \cite{Suu} was removed by \cite{NSU} for smooth coefficients. Since then, \cite{T} extends this result to magnetic potentials lying in $\mathcal C^1$ and \cite{Sa1} extends it to magnetic potentials lying in a Dini class. To our best knowledge, the result with the weakest regularity assumption so far, for general bounded domain, is the one of \cite{KU} where the authors have considered bounded electromagnetic potentials. More recently, in the specific case of a ball in $\R^3$, \cite{H} proved the recovery of unbounded magnetic potentials. Concerning results with partial data associated with this last problem, we mention the work of \cite{Chu,FKSU} and concerning the stability issue, without being exhaustive, we refer to \cite{B,CDR1,CDR2,CP,Pot2,Pot1,Tz}. We mention also the work of \cite{CK,HK,Ki3} related to problems for hyperbolic and parabolic equations treated with an approach similar to the one considered for elliptic equations.

Note that all the above mentioned results have been stated in a bounded domain. Only a small number of articles studied such  inverse boundary value problems in an unbounded domain. In \cite{LU}, the authors combined unique continuation results with  CGO solutions and a Carleman estimate borrowed from \cite{BU} in order to prove the unique recovery of compactly supported electric potentials of a Schr\"odinger operator in a slab  from  partial boundary measurements. This last result has been extended to magnetic Schr\"odinger operators by \cite{KLU} and the stability issue has been addressed by \cite{CM}. We refer also to \cite{Ik,L1,L2,SW,Y} for other related inverse problems stated in a slab. In \cite{CKS2,CKS3}, the authors considered the stable recovery  of coefficients periodic along the axis of an infinite cylindrical domain. More recently, \cite{Ki4} considered, for what seems to be the first time, the recovery of non-compactly supported and non-periodic electric potentials appearing in an infinite cylindrical domain. The results of \cite{Ki4} include also an extension of the work of \cite{LU} to the recovery of non-compactly supported coefficients in a slab.
We mention also the work \cite{BKS,BKS1, CS, KKS,Ki1,  KPS1, KPS2} treating the determination of  coefficients appearing in different PDEs on an infinite cylindrical domain from boundary measurements.

\subsection{Statement of the main results}

Let us recall that there is an obstruction to the simultaneous recovery of $A$, $q$ from the data $\mathcal D_{A,q}$ given by gauge invariance. More precisely according to \cite[Lemma 3.1]{KU}, which is stated for bounded domains but whose arguments can be extended without any difficulty to unbounded domains satisfying \eqref{closed}, the data $\mathcal D_{A,q}$ satisfies the following gauge invariance.
\bel{gauge} \mathcal D_{A+\nabla\phi,q}=\mathcal D_{A,q},\quad \phi\in\{ h_{|\Omega}: \  h\in W^{1,\infty}_{loc}(\R^3:\R),\ \nabla_xh\in L^\infty(\R^3)^3,\ h_{|\R^3\setminus\Omega}=0\}.\ee
 Taking into account this obstruction, for $A=(a_1,a_2,a_3)$,  we consider the recovery of the magnetic field corresponding to the 2-form valued distribution $dA$ defined by
$$dA:=\underset{1\leq j<k\leq 3}{\sum} (\partial_{x_j}a_k-\partial_{x_k}a_j)dx_j\wedge dx_k$$
and $q$. Assuming that $\Omega$ is simply connected and with some suitable regularity assumptions (see for instance Section 4.2), one can check that this result is equivalent  to the recovery of the electromagnetic potential modulo gauge invariance.

 This paper contains three main results. In the first  main result, stated in Theorem \ref{t1}, we consider the unique determination of electromagnetic potentials with low regularity    from the full data $\mathcal D_{A,q}$. In our second main result stated in Theorem \ref{c1}, we prove, for electromagnetic potentials known on the neighborhood of the boundary outside a compact set, that  measurements restricted to a bounded subset of $\pd\Omega$ can also recover uniquely the magnetic field and the electric potential. Finally, in our last result stated in  Theorem \ref{t6}, we give a partial data result by proving the unique recovery of a magnetic field and an electric potential associated with general  class of electromagnetic potentials from restriction of the data $\mathcal D_{A,q}$.

In our first main result we consider general class of bounded electromagnetic potentials and a general closed waveguide. This result can be stated as follows.

\begin{theorem}
\label{t1} 
Let $\Omega$ be an unbounded domain satisfying \eqref{closed},  let $A_1,A_2\in L^\infty(\Omega)^3\cap L^2(\Omega)^3$  be such that $A_1-A_2\in L^1(\Omega)^3$ and let   $q_1,q_2\in L^\infty(\Omega;\mathbb C)$. Then the condition
\bel{t1a} 
\mathcal D_{A_1,q_1}=\mathcal D_{A_2,q_2}
\ee
implies $dA_1=dA_2$. Moreover, assuming that $q_1-q_2\in L^2(\Omega;\mathbb C)$, \eqref{t1a} implies  $q_1=q_2$.
\end{theorem}

Let us remark that Theorem \ref{t1} is stated with boundary measurements in all parts of the unbounded boundary $\pd\Omega$.  Despite the general setting of this problem,  it may be difficult for several applications, like for transmission to long distance,  to have access to such  data. In order to make the measurements more relevant  for some potential applications, we need to consider data restricted to a bounded portion of $\partial\Omega$.
This will be the goal of our second result where we extend  Theorem \ref{t1} to recovery of coefficients from measurements  restricted to  bounded portions of $\pd\Omega$. 
From now on, we
 assume that $\Omega$ is a domain with Lipschitz boundary. For all $s\in\left[0,\frac{1}{2}\right]$, we denote by 
$H^{s}_{loc}(\pd\Omega)$ the set of $f\in L^2_{loc}(\pd\Omega)$  such that for any $\chi\in \mathcal C^\infty_0(\R^3)$, $\chi f \in H^{s}(\pd\Omega)$. For any $u\in H^1(\Omega)$, we can define  $\tau_0 u=u_{|\pd\Omega}$ as an element of $H^{\frac{1}{2}}_{loc}(\pd\Omega)$. In the same way, for $U$ a closed (resp. open) subset of $\pd\Omega$  and for $u\in H^1(\Omega)$ solving $\Delta_Au+qu=0$, with $A\in L^\infty(\Omega)^3$ and $q\in L^\infty(\Omega)$, we denote by  $N_{A,q}u_{|U}$ the restriction of $N_{A,q}u$ to the subspace
$$\{\tau g:\ g\in H^1(\Omega),\ \textrm{supp}(\tau_0g)\subset U\}$$
of $\frac{H^1(\Omega)}{H^1_0(\Omega)}$. Note that here $N_{A,q}u_{|U}$ is the natural extension of the restriction, up to the sign, of the magnetic normal derivative of $u$ to the set $U$.
For $r>0$ and $S_r=\pd\Omega\cap(\overline{\omega}\times[-r,r]) $, we can consider the restriction $\mathcal D_{A,q, r}$ of the data $\mathcal D_{A,q}$ given by
\bel{pada}\mathcal D_{A,q, r}:=\{(\tau u, N_{A,q}u_{|S_r}):\ u\in H^1(\Omega), \textrm{ $u$ solves \eqref{eq1}},\ \textrm{supp}(\tau_0 u)\subset S_r\}.\ee
In the spirit of \cite[Corollary 1.3]{Ki4}, fixing $\delta\in(0,r/2)$, we will apply Theorem \ref{t1} in order to prove the recovery of coefficients known on a neighborhood of the boundary outside $\Omega\cap (\omega\times(\delta-r,r-\delta))$ from the data $\mathcal D_{A,q, r}$. For this purpose we need the following assumption on $\Omega$ and the admissible coefficients.\\
\textbf{Assumption 1:} For $j=1,2$, and for any $F\in L^2(\Omega)$ the equations  $\Delta_{A_j}u_j+\overline{q_j}u_j=F$ and $\Delta_{A_j}u_j+q_ju_j=F$  admit respectively a solution $u_j\in H^1_0(\Omega)$.

We mention that Assumptions 1  will  be fulfilled if for instance $\Omega=\omega_1\times \R$, with $\omega_1$ a bounded open subset of $\R^2$ with Lipschitz boundary, and if $0$ is not in the spectrum of the operators $\Delta_{A_j}+q_j$ and $\Delta_{A_j}+\overline{q_j}$, $j=1,2$, with Dirichlet boundary condition.

Let $\textbf{n}$ be the outward unit normal vector of $\pd\Omega$.\footnote{Since $\Omega$ is only subjected to the condition $\Omega\subset\Omega_1$ we may have $\Omega\neq\Omega_1$ this is the reason why we use a different notation for the outward unit normal vector of $\Omega_1$ and $\Omega$.
} Since $\Omega$ is only subjected to the condition $\Omega\subset\Omega_1$ we may have $\Omega\neq\Omega_1$ this is why we use a different notation for the outward unit normal vector of $\Omega_1$ and $\Omega$.
 Before we state our result, let us also recall that for any $A\in L^\infty(\Omega)^3$ satisfying $\textrm{div}(A)\in L^\infty(\Omega)$, we can define the trace map $A\cdot \textbf{n}$ as the unique element of $$\mathcal B\left(\frac{H^1(\Omega)}{H^1_0(\Omega)}; \left(\frac{H^1(\Omega)}{H^1_0(\Omega)}\right)'\right)$$
defined by
\bel{An}\left\langle (A\cdot \textbf{n})\tau g,\tau h\right\rangle_{\left(\frac{H^1(\Omega)}{H^1_0(\Omega)}\right)',\frac{H^1(\Omega)}{H^1_0(\Omega)}}:=\int_{\Omega}\textrm{div}(A)h\overline{g}dx+\int_{\Omega}A\cdot\nabla h\overline{g}dx+\int_{\Omega} h(A\cdot\overline{\nabla g})dx,\ g,h\in H^1(\Omega).\ee
Again, by a density argument, one can easily check the validity of this definition  by noticing that the right hand side of the identity vanishes as soon as $g\in H^1_0(\Omega)$ or $h\in H^1_0(\Omega)$. Here we use again the fact that, for $u\in H^1(\Omega)$ and $\phi\in\mathcal C^\infty_0(\Omega)$, we have $u\phi\in W^{1,1}_0(\Omega)$. 

Assuming that Assumption 1 is fulfilled, we state our second main result as follows.

\begin{theorem}\label{c1}
Let $\Omega$ be a connected open set with Lipschitz boundary satisfying \eqref{closed}. For $j=1,2$, let $A_j\in L^\infty(\Omega)^3\cap L^2(\Omega)^3$, div$(A_j)\in L^\infty(\Omega)$,  $q_j\in L^\infty(\Omega;\mathbb C)$, $A_1-A_2\in L^1(\Omega)^3$. In addition, let Assumption 1 be fulfilled and, for $A_j\cdot  \textbf{n}$, $j=1,2$,  defined by \eqref{An} with $A=A_j$, let the condition
\bel{c1d}A_1\cdot \textbf{n} =A_2\cdot \textbf{n}\ee
be fulfilled. Assume also that there exist $\delta\in(0,r/2)$ and two open connected set $\Omega_{\pm}\subset\Omega$ with  Lipschitz boundary such that 
\bel{c1a}\pd\Omega\cap (\overline{\omega}\times(-\infty,-r+\delta])\subset \pd\Omega_-,\quad \pd\Omega\cap (\overline{\omega}\times[r-\delta,+\infty))\subset \pd\Omega_+,\ee
\bel{c1b} A_1(x)=A_2(x),\quad q_1(x)=q_2(x),\quad x\in\Omega_-\cup\Omega_+.\ee
Then, the condition 
\bel{c1c} \mathcal D_{A_1,q_1, r}=\mathcal D_{A_2,q_2, r}\ee
implies $dA_1=dA_2$. Moreover, assuming that $q_1-q_2\in L^2(\Omega;\mathbb C)$, \eqref{c1c} implies  $q_1=q_2$.
\end{theorem}

For our last main result we will consider the specific case where $\Omega=\omega\times\R$. This time we want to consider the recovery of the coefficients not from full boundary measurements but from partial boundary measurements without assuming the knowledge of the coefficients close to the boundary. We remark that $\pd\Omega=  \pd \omega\times\R $ and that the outward unit normal vector $\nu$  to $\pd\Omega$ takes the form
$$ \nu(x',x_3)=(\nu'(x'),0)^T,\ x=(x',x_3)\in\pd\Omega, $$
with $\nu'$ the outward unit normal vector of $\pd \omega$. In light of this identity, from now on, we  denote  by $\nu$  both the exterior unit vectors normal to $\pd  \omega$ and to $\pd \omega\times\R$.
We fix $\theta_0 \in \mathbb{S}^1 :=\{ y\in\R^2;\ \abs{y}=1\}$ and  we introduce the $\theta_0$-illuminated (resp., $\theta_0$-shadowed) face of $\pd \omega$, defined by
$$\pd \omega_{\theta_0}^- := \{ x \in \pd \omega;\ \theta_0 \cdot \nu(x) \leq 0 \}\ (\mbox{resp.},\ \pd \omega_{\theta_0}^+= \{x \in \pd \omega;\ \theta_0 \cdot \nu(x) \geq 0\}).$$
From now on, we denote by 
$x \cdot y := \sum_{j=1}^k x_j y_j$
the Euclidian scalar product of any two vectors $x:=(x_1,\ldots,x_k)^T$ and $y:=(y_1,\ldots,y_k)^T$ of $\mathbb C^k$.  We fix $V$ a portion of $\pd\Omega$ taking the form $V:=V'\times \R$, where $V'$ 
is an arbitrary open neighborhood of  $\pd \omega_{\theta_0}^-$ in $\partial\omega$. We introduce also the set of data
$$\mathcal D_{A,q, V}=\{(\tau u, N_{A,q}u_{|V}):\ u\in H^1(\Omega), \textrm{ $u$ solves \eqref{eq1}}\}.$$
Then we can state our last main result as follows.

\begin{theorem}
\label{t6} 
Let $\Omega=\omega\times\R$ and, for $j=1,2$, let $A_j\in L^\infty(\Omega)^3\cap L^2(\Omega)^3$,  div$(A_j)\in L^\infty(\Omega)$,  $q_j\in L^\infty(\Omega;\mathbb C)$, $A_1-A_2\in L^1(\Omega)^3$.  Let also  $A_1$ and $A_2$  satisfy \eqref{c1d}. Then the condition
\bel{t6a} 
\mathcal D_{A_1,q_1, V}=\mathcal D_{A_2,q_2,V}
\ee
implies $dA_1=dA_2$. Moreover,   assuming that  $q_1-q_2\in L^1(\Omega;\mathbb C)$, \eqref{t1a} implies also that   $q_1=q_2$.
\end{theorem}
\subsection{Comments about our results}
To the best of our knowledge Theorem \ref{t1} is the first result of recovery of a magnetic field and an electric potential in an unbounded domain with such a general setting. This point can be seen  through four different aspects of the theorem. First, Theorem \ref{t1} is stated in a general unbounded domain subject only to condition \eqref{closed}. This makes an important difference with other related results which, to our best knowledge, have all been stated in specific unbounded domains like a slab, the half space or a cylindrical domain (see \cite{KLU,LU,CKS2,CKS3}). In particular, Theorem \ref{t1} holds true with domains having different types of geometrical deformations like bends or twisting, which are frequently used in problems of  transmission for improving the propagation. Second, to the best of our knowledge,  in contrast to all other results stated for elliptic equations in an unbounded domain, Theorem \ref{t1} requires no assumptions about the spectrum of the magnetic Schr\"odinger operator associated with the electromagnetic potential under consideration. Usually such conditions make some restrictions on the class of coefficients under consideration, here we avoid such constraints. Third, we prove, for what seems to be the first time, the recovery of electromagnetic potentials that are neither compactly supported nor periodic. Actually we consider a class of electromagnetic potentials admitting various type of behavior outside a compact set (roughly speaking we consider magnetic potentials lying in $L^1(\Omega)^3$ and electric potentials lying in $L^2(\Omega)$).  Fourth, Theorem \ref{t1} seems to be the first result stated for an unbounded domain with electromagnetic potentials having regularity comparable to \cite{KU}, where the recovery of electromagnetic potentials has been stated with the weakest regularity condition so far for general bounded domains. 

The main tools in our analysis are CGO solutions suitably designed for unbounded domains satisfying \eqref{closed}. Here in contrast to  \cite{CKS2,CKS3,KLU,LU} we do not restrict our analysis to  compactly supported or periodic coefficients where, by mean of unique continuation or Floquet decomposition,  one can transform the problem stated  on an unbounded domain into a problem on a bounded domain. Like \cite{Ki4}, we introduce a new class of CGO solutions designed for infinite cylindrical domains.  The difficulties in the construction of such solutions are coming both from the fact that we consider magnetic potentials that are not compactly supported and the fact that we need to preserve the square integrability of the CGO solutions, which is not guarantied by the usual CGO solutions in unbounded domains. In addition, like in  \cite{KU}, we  build CGO solutions designed for bounded magnetic potentials. The construction of our CGO solutions requires   Carleman estimates in negative order Sobolev space that we prove by extending some results, similar to those of \cite{FKSU,ST}, to infinite cylindrical domains.

Let us observe that the construction of  CGO solutions satisfying the square integrability property works only for  domains contained into an infinite cylinder. For instance, we can not apply our construction to domains like slab or half space. However, in a similar way to \cite[Corollary 1.4]{Ki4}, applying Theorem \ref{t1} and  \ref{c1}, one can prove that the result of \cite{KLU} can be extended to electromagnetic potentials supported in infinite cylinder.

In this paper we consider electric potentials $q$ that can be complex valued but we consider magnetic potentials $A$ that take value in $\R^3$. Like in \cite{KLU,KU}, we could state our result with  magnetic potentials taking value in $\mathbb C^3$, but for simplicity we restrict our analysis  to real valued magnetic potentials.

\subsection{Outline}
This paper is organized as follows. In Section 2, we derive some Carleman estimates that will be useful at the same time for building the CGO solutions and restricting the data  in Theorem \ref{t6}.   In Section 3, we use the Carleman estimates in order to build our CGO solutions. Combining all these tools, in Section 4, 5, 6  we prove respectively Theorem \ref{t1}, Theorem \ref{c1} and Theorem \ref{t6}. Finally, in Section 7 we explain how our result can be extended to higher dimension.

\section{Carleman estimates}
From now on, we fix $\Omega_1=\omega\times\R$. We associate to every point $x \in \Omega_1$ the coordinates $x=(x',x_3)$, where $x_3 \in \R$ and $x':= (x_1,x_2) \in \omega$. In a similar way to the discussion before the statement of Theorem \ref{t6}, we  denote  by $\nu$  both the exterior unit vectors normal to $\pd  \omega$ and to $\pd\Omega_1$. The goal of this section is  to establish two  Carleman estimates for the magnetic Laplace operator in the unbounded cylindrical domain $\Omega_1$. 
We start with a  Carleman estimate which will be our first main tool. Then, using this Carleman estimate we will derive a Carleman estimate in negative order Sobolev space.

\subsection{General Carleman estimate}

In order to prove our Carleman estimates we introduce first a weight function depending on two parameters $s,\rho\in(1,+\infty)$ and we consider, for $\rho>s>1$ and $\theta\in\mathbb S^2$, the perturbed weight
\bel{phi}\phi_{\pm,s}(x',x_3):=\pm \rho\theta\cdot x'-s{(x'\cdot\theta)^2\over 2},\quad x=(x',x_3)\in\omega\times\R=\Omega_1.\ee
We define
\[ P_{A,q,\pm,s}:=e^{-\phi_{\pm,s}}(\Delta +2iA\cdot\nabla +q)e^{\phi_{\pm,s}}.\]
Like in \cite{FKSU,ST}, we consider convexified weight, instead of the linear weight used in \cite[Proposition 31]{Ki4}, in order to be able to absorb first order perturbations of the Laplacian.
Our first Carleman estimates can be seen as an extension of \cite[Proposition 2.3]{FKSU}, stated with linear weight, to unbounded cylindrical domains.  These estimates
take the following form. 
\begin{proposition}\label{p1} Let $A\in L^\infty(\Omega_1)^3$ and $q\in L^\infty(\Omega_1;\mathbb C)$. Then  there exist $s_1>1$ and, for $s>s_1$,  $\rho_1(s)$ such that for any $v\in\mathcal C^2_0(\R^3)\cap H^1_0(\Omega_1)$ 
the estimate
\bel{p1a} \begin{aligned}&\rho\int_{\partial\omega_{\pm,\theta}\times\R} |\partial_\nu v|^2|\theta\cdot \nu| d\sigma(x)+s\rho^{-2}\int_{\Omega_1}|\Delta v|^2dx+ s\int_{\Omega_1}|\nabla v|^2dx+s\rho^2\int_{\Omega_1}|v|^2dx \\
&\leq C\left[\norm{P_{A,q,\pm,s}v}^2_{L^2(\Omega_1)}+\rho\int_{\partial\omega_{\mp,\theta}\times\R} |\partial_\nu v|^2|\theta\cdot \nu| d\sigma(x)\right]\end{aligned}\ee
holds true for $s>s_1$, $\rho\geq \rho_1(s)$  with $C$  depending only on  $\Omega_1$ and $M\geq \norm{q}_{L^\infty(\Omega_1))}+\norm{A}_{L^\infty(\Omega_1)^3}$.
 \end{proposition}
\begin{proof}
We start by proving that  for all $s>1$ there exists $\rho_1(s)$ such that for $\rho >\rho_1(s)$  we have
\bel{p1b}\begin{aligned}\norm{e^{-\phi_{\pm,s}}\Delta e^{\phi_{\pm,s}}v}^2_{L^2(\Omega_1)}\geq&\rho\int_{\partial\omega_{\pm,\theta}\times\R} |\partial_\nu v|^2|\theta\cdot \nu| d\sigma(x)-8\rho\int_{\partial\omega_{\mp,\theta}\times\R} |\partial_\nu v|^2|\theta\cdot \nu| d\sigma(x)+s\int_{\Omega_1}|\nabla v|^2dx\\
\ &+\frac{s\rho^2}{2}\int_{\Omega_1}|v|^2dx+cs\rho^{-2}\int_{\Omega_1}|\Delta v|^2dx,\end{aligned}\ee
with $c$ depending only on $\Omega_1$.
Using this estimate,  we will derive \eqref{p1a}. The proof of this result being similar for $e^{-\phi_{+,s}}\Delta e^{\phi_{+,s}}$ and $e^{-\phi_{-,s}}\Delta e^{\phi_{-,s}}$, we will only consider it for $e^{-\phi_{+,s}}\Delta e^{\phi_{+,s}}$.  We decompose $e^{-\phi_{+,s}}\Delta e^{\phi_{+,s}}$ into three terms
\[e^{-\phi_{+,s}}\Delta e^{\phi_{+,s}}=P_{1,+}+P_{2,+}+P_{3,+},\]
with
\[P_{1,+}=\Delta'+|\nabla\phi_{+,s}|^2-\Delta' \phi_{+,s}=\Delta'+\rho^2-2s\rho (x'\cdot\theta)+s^2 (x'\cdot\theta)^2+s,\]

\[P_{2,+}=\partial_{x_3}^2,\quad P_{3,+}=2\nabla'\phi_{+,s}\cdot\nabla' +2\Delta' \phi_{+,s}=2(\rho-s (x'\cdot\theta))\theta  \cdot\nabla'-2s.\]
Here $\Delta':=\pd_{x_1}^2+\pd_{x_2}^2$, $\nabla':=(\pd_{x_1},\pd_{x_2})^T$ and $\theta\cdot\nabla'=\theta_1\pd_{x_1}+\theta_2\pd_{x_2}$.
Using some arguments similar to  \cite[Proposition 2.3]{FKSU}, one can check that for all $s>1$ there exists $\rho_2(s)>1$ such that for $\rho >\rho_2(s)$  and $y\in\mathcal C^\infty(\overline{\omega})\cap H^1_0(\omega)$ we have
$$\begin{aligned}&2\re\int_\omega P_{1,+}y\overline{P_{3,+}y}dx'\\
&\geq \rho\int_{\partial\omega_{\pm,\theta}} |\partial_\nu y|^2|\theta\cdot \nu| d\sigma(x')-8\rho\int_{\partial\omega_{\mp,\theta}} |\partial_\nu y|^2|\theta\cdot \nu| d\sigma(x')+s\rho^2\int_{\Omega_1}|y|^2dx'+s\int_\omega |\nabla' y|^2dx.\end{aligned}$$
 Applying this estimate to $v(\cdot,x_3):=x'\mapsto v(x',x_3)$, $x_3\in\R$, we obtain
$$\begin{aligned}2\re\int_\omega P_{1,+}v(\cdot,x_3)\overline{P_{3,+}v(\cdot,x_3)}dx'&\geq \rho\int_{\partial\omega_{\pm,\theta}} |\partial_\nu v(\cdot,x_3)|^2|\theta\cdot \nu| d\sigma(x')+s\int_\omega |\nabla' v(\cdot,x_3)|^2dx\\
\ &\ \ \ -8\rho\int_{\partial\omega_{\mp,\theta}} |\partial_\nu v(\cdot,x_3)|^2|\theta\cdot \nu| d\sigma(x')+s\rho^2\int_{\omega}|v(\cdot,x_3)|^2dx',\quad x_3\in\R.\end{aligned}$$
Integrating this estimate with respect to $x_3\in\R$, we get
\bel{p1c}\begin{aligned}&\norm{P_{1,+}v+P_{2,+}v+P_{3,+}v}^2_{L^2(\Omega_1)}\\
&\geq \norm{P_{1,+}v+P_{2,+}v}^2_{L^2(\Omega_1)}+2\re\int_{\Omega_1} P_{1,+}v\overline{P_{3,+}v}dx+2\re\int_{\Omega_1} P_{2,+}v\overline{P_{3,+}v}dx \\
&\geq \norm{P_{1,+}v+P_{2,+}v}^2_{L^2(\Omega_1)}+2\re\int_{\Omega_1} P_{2,+}v\overline{P_{3,+}v}dx+2\rho\int_{\pd\omega_{+,\theta}\times\R} |\partial_\nu v|^2|\theta\cdot \nu| d\sigma(x)  \\
&\ \ \ -8\rho\int_{\pd\omega_{-,\theta}\times\R} |\partial_\nu v|^2|\theta\cdot \nu| d\sigma(x)+s\rho^2\int_{\Omega_1} |v|^2dx+s\int_{\Omega_1} |\nabla' v|^2dx .\end{aligned}\ee
On the other hand, integrating by parts with respect to $x_3\in\R$ and then with respect to $x'\in\omega$, we find
\bel{p1d}\begin{aligned}\re\int_{\Omega_1} P_{2,+}v\overline{P_{3,+}v}dx&=-\int_{\Omega_1}(\rho-s (x'\cdot\theta))\theta  \cdot\nabla'|\partial_{x_3}v|^2dx+2s\int_{\Omega_1}|\partial_{x_3}v|^2dx\\
\ &=s\int_{\Omega_1}|\partial_{x_3}v|^2dx.\end{aligned}\ee
Moreover, fixing
$$\tilde{c}=4\left(3+\sup_{x'\in\overline{\omega}}|x'|\right)^2,\quad \rho_1(s)=\rho_2(s)+\tilde{c}^{-1}\sqrt{s},$$
we deduce that, for $\rho>\rho_1(s)$, we have
$$\norm{P_{1,+}v+P_{2,+}v}^2_{L^2(\Omega_1)}\geq s\tilde{c}^{-1}\rho^{-2}\norm{P_{1,+}v+P_{2,+}v}^2_{L^2(\Omega_1)}\geq s(2\tilde{c})^{-1}\rho^{-2}\norm{\Delta v}_{L^2(\Omega_1)}^2-\frac{s\rho^2}{2} \norm{ v}_{L^2(\Omega_1)}^2.$$
Combining this with \eqref{p1c}-\eqref{p1d} we deduce \eqref{p1b}. Now let us complete the proof of \eqref{p1a}. For this purpose, we introduce 
$$P_{4,\pm}=2iA\cdot\nabla+2iA\cdot\nabla \phi_{\pm,s}+q=2iA\cdot\nabla+2(\pm\rho-s(x'\cdot\theta)) iA'\cdot\theta +q,$$
with $A=(a_1,a_2,a_3)$ and $A'=(a_1,a_2)$, and we recall that $P_{A,q,\pm,s}=e^{-\phi_{\pm,s}}\Delta e^{\phi_{\pm,s}}+P_{4,\pm}$. We find
\[\begin{aligned}&\norm{P_{A,q,\pm,s}v}^2_{L^2(\Omega_1)}\\
&\geq {\norm{e^{-\phi_{\pm,s}}\Delta e^{\phi_{\pm,s}}v}^2_{L^2(\Omega_1)}\over 2}-\norm{P_{4,\pm}v}_{L^2(\Omega_1)}^2\\
&\geq {\norm{^{-\phi_{\pm,s}}\Delta e^{\phi_{\pm,s}}v}^2_{L^2(\Omega_1)}\over 2}-3\norm{A}_{L^\infty(\Omega_1)}^2\int_{\Omega_1} |\nabla  v|^2dx -3\left(16\norm{A}_{L^\infty(\Omega_1)}^2\rho +\norm{q}_{L^\infty(\Omega_1)}^2\right)\int_{\Omega_1} | v|^2dx .\end{aligned}\]
Fixing $s_1=48\norm{A}_{L^\infty(\Omega_1)}^2+6$, we deduce \eqref{p1a} from \eqref{p1b}.
\end{proof}

A direct consequence of these Carleman estimates  is the following result which will be useful for Theorem \ref{t6}.

\begin{co}\label{c2}  Let $A\in L^\infty(\Omega_1)^3$ and $q\in L^\infty(\Omega_1;\mathbb C)$.  There exists $\rho_1'>0$ such that for  any  $u\in\mathcal C^2_0(\R^3)\cap H^1_0(\Omega_1)$ 
the estimate
\begin{equation}\label{c2a}\begin{array}{l}\rho\int_{\pd\omega_{+,\theta}\times\R}e^{-2\rho\theta\cdot x'}\abs{\partial_\nu u}^2\abs{\theta\cdot\nu(x) } d\sigma(x)
+\rho^2\int_{\Omega_1} e^{-2\rho\theta\cdot x'}\abs{u}^2dx+\int_{\Omega_1} e^{-2\rho\theta\cdot x'}|\nabla u|^2dx\\
\leq C\left(\int_{\Omega_1} e^{-2\theta\cdot x'}\abs{(-\Delta+2iA\cdot\nabla+q)u}^2dx+\rho\int_{\pd\omega_{-,\theta}\times\R}e^{-2\rho\theta\cdot x'}\abs{\partial_\nu u}^2\abs{\theta\cdot\nu(x) }d\sigma(x)\right)\end{array}\end{equation}
holds true for $\rho\geq \rho_1'$  with $C$  depending only on  $\Omega_1$ and $M\geq \norm{q}_{L^\infty(\Omega_1)}+\norm{A}_{L^\infty(\Omega_1)^3}$.
\end{co}
\begin{proof}
We fix $u\in\mathcal C^2_0(\R^3)\cap H^1_0(\Omega_1)$   and we set  $v=e^{-\phi_{+,s}}u$ such that $$\int_{\Omega_1} e^{-2\phi_{+,s}}|(-\Delta+2iA\cdot\nabla+q)u|^2dx=\int_{\Omega_1}|P_{A,q,+,s}v|^2dx.$$
The fact that $v\in H^1_0(\Omega_1)$ implies $\partial_\nu v_{\vert\pd\Omega_1}=e^{-\rho\theta\cdot x'}e^{{s(x\cdot\theta)^2\over2}}\partial_\nu u_{\vert\pd\Omega_1}$ and we deduce that
\bel{c2b}\int_{\pd\omega_{+,\theta}\times\R} |\pd_\nu v|^2\omega\cdot\nu d\sigma(x)\geq \int_{\pd\omega_{+,\theta}\times\R} e^{-2\rho\theta\cdot x'}|\pd_\nu u|^2\omega\cdot\nu d\sigma(x)\ee
\bel{c2c}\int_{\pd\omega_-\times\R} |\pd_\nu v|^2\omega\cdot\nu d\sigma(x)\geq e^{sb^2}\int_{\pd\omega_-\times\R} e^{-2\rho\theta\cdot x'}|\pd_\nu u|^2\omega\cdot\nu d\sigma(x),\ee
with $b=\left(2+2\sup_{x'\in\omega}|x'|\right)$.
Moreover, since
\[ \nabla u(x)=\nabla (e^{\phi_{+,s}}v)=(\rho-sx'\cdot\theta)u\omega+e^{\rho\theta\cdot x'}e^{-{s(x'\cdot\theta)^2\over2}}\nabla v,\quad x=(x',x_3)\in\omega\times\R,\]
we obtain
\[\int_{\Omega_1} e^{-2\rho\theta\cdot x'}|\nabla u|^2dx\leq 2\rho^2e^{sb^2}\int_{\Omega_1} |v|^2dx+2e^{sb^2}\int_{\Omega_1}|\nabla v|^2dx.\]
Combining this estimates with \eqref{p1a} and \eqref{c2b}-\eqref{c2c}, for $s\geq s_1$ and $\rho>\rho_1(s)$, we get
\bel{10}\begin{array}{l}\int_{\Omega_1} e^{-2\rho\theta\cdot x'}|\nabla u|^2dx+\rho^2\int_{\Omega_1} e^{-2\rho\theta\cdot x'}|u|^2dx+\rho\int_{\pd\omega_{+,\theta}\times\R} e^{-2\rho\theta\cdot x'}|\pd_\nu u|^2\omega\cdot\nu d\sigma(x)\\
\ \\
\leq \rho e^{sb^2}\int_{\pd\omega_{-,\theta}\times\R} e^{-2\rho\theta\cdot x'}|\pd_\nu u|^2\omega\cdot\nu d\sigma(x)+Ce^{sb^2}\int_{\Omega_1} e^{-2\rho\theta\cdot x'}|(-\Delta+2iA\cdot\nabla+q)u|^2dx.\end{array}\ee
From this last estimate we deduce \eqref{c2a} by fixing $s=s_1+1$ and $\rho_1'=\rho_1(s_1+1)$. 

\end{proof}

\begin{rem}\label{r1} By density the result of Proposition \ref{p1} and Corollary \ref{c1} can be extended to any $v\in H^1_0(\Omega_1)$ satisfying $\Delta v\in L^2(\Omega_1)$ and $\partial_\nu v\in L^2(\pd\Omega_1)$.\end{rem}
\subsection{Carleman estimate in negative order Sobolev space}

The goal of this subsection is to apply the result of Proposition \ref{p1} in order to derive  Carleman estimates in negative order Sobolev space which will be   one of the most important ingredient  in the construction of the CGO solutions.
We recall first some preliminary tools and we derive a Carleman estimate in Sobolev space of negative order.  In a similar way to \cite{Ki3}, for all $m\in\R$, we introduce the space $H^m_\rho(\R^{3})$ defined by
\[H^m_\rho(\R^{3})=\{u\in\mathcal S'(\R^{3}):\ (|\xi|^2+\rho^2)^{m\over 2}\hat{u}\in L^2(\R^{3})\},\]
with the norm
\[\norm{u}_{H^m_\rho(\R^{3})}^2=\int_{\R^3}(|\xi|^2+\rho^2)^{m}|\hat{u}(\xi)|^2 d\xi .\]
Here for all tempered distributions $u\in \mathcal S'(\R^3)$, we denote by $\hat{u}$ the Fourier transform of $u$ which, for $u\in L^1(\R^3)$, is defined by
$$\hat{u}(\xi):=\mathcal Fu(\xi):= (2\pi)^{-{3\over2}}\int_{\R^3}e^{-ix\cdot \xi}u(x)dx.$$
From now on, for $m\in\R$ and $\xi\in \R^3$,  we set $$\left\langle \xi,\rho\right\rangle=(|\xi|^2+\rho^2)^{1\over2}$$
and $\left\langle D_x,\rho\right\rangle^m u$ defined by
\[\left\langle D_x,\rho\right\rangle^m u=\mathcal F^{-1}(\left\langle \xi,\rho\right\rangle^m \mathcal Fu).\]
For $m\in\R$ we define also the class of symbols
\[S^m_\rho=\{c_\rho\in\mathcal C^\infty(\R^3\times\R^3):\ |\pd_x^\alpha\pd_\xi^\beta c_\rho(x,\xi)|\leq C_{\alpha,\beta}\left\langle \xi,\rho\right\rangle^{m-|\beta|},\  \alpha,\beta\in\mathbb N^3\}.\]
Following \cite[Theorem 18.1.6]{Ho3}, for any $m\in\R$ and $c_\rho\in S^m_\rho$, we define $c_\rho(x,D_x)$, with  $D_x=-i\nabla $, by
\[c_\rho(x,D_x)y(x)=(2\pi)^{-{3\over 2}}\int_{\R^3}c_\rho(x,\xi)\hat{y}(\xi)e^{ix\cdot \xi} d\xi,\quad y\in\mathcal S(\R^3).\]
For all $m\in\R$, we set also $OpS^m_\rho:=\{c_\rho(x,D_x):\ c_\rho\in S^m_\rho\}$.
We fix
$$P_{A,q,\pm}:=e^{\mp \rho x'\cdot\theta}(\Delta_{ A}+q) e^{\pm \rho x'\cdot\theta}$$
and, in the spirit of \cite[estimate (2.14)]{FKSU} and \cite[Lemma 2.1]{ST},  we consider the following Carleman estimate.

\begin{proposition}\label{p2} Let $A\in L^\infty(\Omega_1)^3$ and $q\in L^\infty(\Omega_1;\mathbb C)$. Then, there exists $\rho_2>1$  such that  for all $v\in \mathcal C^\infty_0(\Omega_1)$, we have 
\bel{p2a}\rho^{-1}\norm{v}_{H^1_\rho(\R^3)}\leq C\norm{P_{A,q,\pm}v}_{H^{-1}_\rho(\R^3)},\quad \rho>\rho_2,\ee
with $C>0$ depending on $\Omega_1$  and $\norm{q}_{L^\infty(\Omega_1)}+\norm{A}_{L^\infty(\Omega_1)^3}$.
\end{proposition}
\begin{proof} 
Since this result is similar for $P_{A,q,+}v$ and $P_{A,q,-}v$, we will only prove it for $P_{A,q,+}v$.
For $\phi_{+,s}$ given by \eqref{phi}, we  consider
$$R_{ A,q,+,s}:=e^{-\phi_{+,s}}(\Delta_{ A}+q)e^{\phi_{+,s}}$$
and in a similar way to Proposition \ref{p1} we decompose $R_{A,+,s}$ into three terms
\[R_{ A,q,+,s}=P_{1,+}+P_{2,+}+P_{3,+,A},\]
where we recall that
\[P_{1,+}=\Delta+\rho^2-2s\rho (x'\cdot\theta)+s^2 (x'\cdot\theta)^2+s,\quad  P_{2,+}=2(\rho-s (x'\cdot\theta))\theta  \cdot\nabla-2s.\]
\[ P_{3,+,A}=2iA\cdot\nabla+2iA\cdot\nabla \phi_{+,s}+q-|A|^2+i\textrm{div}(A)=2iA\cdot\nabla+2(\rho-s(x'\cdot\theta)) iA'\cdot\theta +q-|A|^2+i\textrm{div}(A).\]
We pick $ \tilde{\omega}$  a bounded $\mathcal C^2$ open set of $\R^2$ such that $\overline{\omega}\subset\tilde{\omega}$ and we extend the function  $A$ and $q$ to $\R^3$ with  $ A=0$, $q=0$ on $\R^3\setminus \Omega_1$. We consider also $\tilde{\Omega}=\tilde{\omega}\times\R$. We start with the Carleman estimate
\bel{car}\rho^{-1}\norm{v}_{H^1_\rho(\R^3)}\leq C\norm{R_{ A,q,+,s}v}_{H^{-1}_\rho(\R^3)},\quad v\in\mathcal C^\infty_0(\Omega_1).\ee
For this purpose, we fix $w\in H^3(\R^3)$ satisfying supp$(w)\subset\tilde{\Omega}$  and we  consider the quantity
\[\left\langle D_x,\rho\right\rangle^{-1}(P_{1,+}+P_{2,+})\left\langle D_x,\rho\right\rangle w.\]
In all the remaining parts of this proof $C>0$ denotes a generic constant depending on $\Omega_1$ and $\norm{A}_{L^\infty(\Omega_1)^3}+\norm{q}_{L^\infty(\Omega_1)}$.
Applying the properties of composition of pseudoddifferential operators (e.g. \cite[Theorem 18.1.8]{Ho3}), we find
\bel{l2c}\left\langle D_x,\rho\right\rangle^{-1}(P_{1,+}+P_{2,+})\left\langle D_x,\rho\right\rangle=P_{1,+}+P_{2,+}+S_\rho(x,D_x),\ee
where $S_\rho$ is defined by
\[S_\rho(x,\xi)=\nabla_\xi\left\langle \xi,\rho\right\rangle^{-1}\cdot D_x(p_{1,+}(x,\xi)+p_{2,+}(x,\xi))\left\langle \xi,\rho\right\rangle+\underset{\left\langle \xi,\rho\right\rangle\to+\infty}{ o}(1),\]
with
$$p_{1,+}(x,\xi)=-|\xi|^2+\rho^2-2s\rho (x'\cdot\theta)+s^2 (x'\cdot\theta)^2+s,\quad p_{2,+}(x,\xi)=2i[\rho-s(x'\cdot\theta)]\theta\cdot\xi'-2s,\quad \xi=(\xi',\xi_3)\in \R^2\times\R.$$
Therefore, we have
$$S_\rho(x,\xi)={[-2i\rho s+2is^2x'\cdot\theta+2s(\theta\cdot\xi')](\theta\cdot\xi')\over |\xi|^2+\rho^2}+\underset{\left\langle \xi,\rho\right\rangle\to+\infty}{ o}(1)$$
and it follows
\bel{l2d} \norm{S_\rho(x,D_x)w}_{L^2( \R^3)}\leq Cs^2\norm{w}_{L^2( \R^3)}.\ee
On the other hand,   applying \eqref{p1a} to $w$, which is permitted according to Remark \ref{r1}, with $\Omega_1$ replaced by  $\tilde{\Omega}$ and $A=0$, $q=0$,  we get
\[\norm{P_{1,+}w+P_{2,+}w}_{L^2(\R^3)}\geq C\left(s^{1/2}\rho^{-1}\norm{\Delta w}_{L^2(\R^3)}+s^{1/2}\norm{\nabla w}_{L^2(\R^3)}+s^{1/2}\rho\norm{ w}_{L^2(\R^3)}\right).\]
Combining this estimate with \eqref{l2c}-\eqref{l2d}, for ${\rho\over s^2}$ sufficiently large, we obtain
$$\begin{array}{l} \norm{(P_{1,+}+P_{2,+})\left\langle D_x,\rho\right\rangle w}_{H^{-1}_\rho(\R^3)}\\
=\norm{\left\langle D_x,\rho\right\rangle^{-1}(P_{1,+}+P_{2,+})\left\langle D_x,\rho\right\rangle w}_{L^2( \R^3)}\\ \geq  Cs^{1/2}\left(\rho^{-1}\norm{\Delta w}_{L^2(\R^3)}+\norm{\nabla w}_{L^2(\R^3)}+\rho\norm{ w}_{L^2(\R^3)}\right).\end{array}$$
On the other hand, using the fact that $w\in H^2(\tilde{\Omega})\cap H^1_0(\tilde{\Omega})$, the elliptic regularity for cylindrical domain (e.g. \cite[Lemma 2.2]{CKS}) implies
$$\norm{w}_{H^2(\R^3)}=\norm{w}_{H^2(\tilde{\Omega})}\leq C(\norm{\Delta w}_{L^2(\tilde{\Omega})}+\norm{ w}_{L^2(\tilde{\Omega})}).$$
Combining this with the previous estimate, for $s$ sufficiently large, we find
 \bel{l2e}\norm{(P_{1,+}+P_{2,+})\left\langle D_x,\rho\right\rangle w}_{H^{-1}_\rho(\R^3)}\geq Cs^{\frac{1}{2}}\rho^{-1}\norm{w}_{H^2_\rho(\R^3)}.\ee
Moreover, we have
\bel{p2c}\begin{aligned}&\norm{P_{3,+,A}\left\langle D_x,\rho\right\rangle w}_{H^{-1}_\rho(\R^3)}\\
&\leq \norm{[2i(\rho-s(x'\cdot\theta))A\cdot\theta+(q-|A|^2)]\left\langle D_x,\rho\right\rangle w}_{H^{-1}_\rho(\R^3)}+2\norm{A\cdot\nabla \left\langle D_x,\rho\right\rangle w}_{H^{-1}_\rho(\R^3)}\\
&\ \ \ \ +\norm{i\textrm{div}(A)\left\langle D_x,\rho\right\rangle w}_{H^{-1}_\rho(\R^3)}.\end{aligned}\ee
For the first term on the right hand side of this inequality, we have
\bel{p2d}\begin{aligned}\norm{[2i(\rho-s(x'\cdot\theta))A\cdot\theta+(q-|A|^2)]\left\langle D_x,\rho\right\rangle w}_{H^{-1}_\rho(\R^3)}&\leq\rho^{-1}\norm{[2i(\rho-s(x'\cdot\theta))A\cdot\theta+(q-|A|^2)]\left\langle D_x,\rho\right\rangle w}_{L^2(\R^3)}\\
\ &\leq C\norm{\left\langle D_x,\rho\right\rangle w}_{L^2(\R^3)}\\
\ &\leq C\norm{\left\langle D_x,\rho\right\rangle w}_{L^2(\R^3)}=C\norm{ w}_{H^1_\rho(\R^3)},\end{aligned}\ee
with $C$ depending only on $\norm{A}_{L^\infty(\Omega_1)^3}+\norm{q}_{L^\infty(\Omega_1)}$. For the second term on the right hand side of \eqref{p2c}, we get
\bel{p2e}\begin{aligned}\norm{A\cdot\nabla \left\langle D,\rho\right\rangle w}_{H^{-1}_\rho(\R^3)}&\leq \rho^{-1}\norm{A\cdot\nabla \left\langle D_x,\rho\right\rangle w}_{L^2(\R^3)}\\
\ &\leq \rho^{-1}\norm{A}_{L^\infty(\Omega_1)^3}\norm{\nabla \left\langle D,\rho\right\rangle w}_{L^2(\R^3)}\\
\ &\leq \rho^{-1}\norm{A}_{L^\infty(\Omega_1)^3}\norm{ w}_{H^{2}_\rho(\R^3)}.\end{aligned}\ee
Finally,  for the last term on the right hand side of \eqref{p2c}, by duality, we find
\bel{p2f}\begin{aligned}\norm{i\textrm{div}(A)\left\langle D_x,\rho\right\rangle w}_{H^{-1}_\rho(\R^3)}&\leq \rho^{-1} \norm{A\cdot\nabla \left\langle D_x,\rho\right\rangle w}_{L^2(\R^3)}+\norm{(\left\langle D_x,\rho\right\rangle w) A}_{L^2(\R^3)^3}\\
\ &\leq 2\rho^{-1}\norm{A}_{L^\infty(\Omega_1)^3}\norm{w}_{H^2_\rho(\R^3))}.\end{aligned}\ee
Combining \eqref{p2c}-\eqref{p2f}, we obtain
$$\norm{P_{3,+,A}\left\langle D_x,\rho\right\rangle w}_{H^{-1}_\rho(\R^3)}\leq C\rho^{-1}\norm{w}_{H^2_\rho(\R^3)}$$
and combining this with \eqref{l2e} for $s>1$ sufficiently large, we get
\bel{p2g}\norm{R_{ A,q,+,s}\left\langle D_x,\rho\right\rangle w}_{H^{-1}_\rho(\R^3)}^2\geq Cs^{\frac{1}{2}}\rho^{-1}\norm{w}_{H^2_\rho(\R^3)}.\ee
Now let us set $\omega_j$, $j=1,2$ two open subsets of $\tilde{\omega}$  such that $\overline{\omega}\subset \omega_1$, $\overline{\omega_1}\subset \omega_2$, $\overline{\omega_2}\subset \tilde{\omega}$.
We fix $\psi_0\in\mathcal C^\infty_0(\tilde{\omega})$ satisfying $\psi_0=1$ on $\overline{\omega_2}$, $w(x',x_3)=\psi_0(x') \left\langle D_x,\rho\right\rangle^{-1} v(x',x_3)$ and for $\psi_1\in\mathcal C^\infty_0(\omega_1)$ satisfying $\psi_1=1$ on $\omega$, we get $$(1-\psi_0 )\left\langle D_x,\rho\right\rangle^{-1} v=(1-\psi_0 )\left\langle D_x,\rho\right\rangle^{-1}\psi_1 v,$$
where $\psi_1v$ denotes the function $(x',x_3)=x\mapsto \psi_1(x')v(x)$. According to \cite[Theorem 18.1.8]{Ho3}, since $1-\psi_0$ is vanishing in a neighborhood of supp$(\psi_1)$, we have $(1-\psi_0) \left\langle D_x,\rho\right\rangle^{-1}\psi_1\in OpS^{-\infty}_\rho$ and it follows
\[ \begin{aligned}\rho^{-1}\norm{v}_{H^{1}_\rho(\R^3)}&=\rho^{-1}\norm{\left\langle D_x,\rho\right\rangle^{-1} v}_{H^2_\rho(\R^3)}\\
\ &\leq \rho^{-1}\norm{w}_{H^2_\rho(\R^3)}+\rho^{-1}\norm{(1-\psi_0)\left\langle D_x,\rho\right\rangle^{-1}\psi_1 v}_{H^2_\rho(\R^3)}\\
\ &\leq \rho^{-1}\norm{w}_{H^2_\rho(\R^3)}+{C\norm{v}_{L^2(\R^3)}\over\rho^2} .\end{aligned}\]
In the same way, we find
$$\begin{aligned}\norm{P_{A,-,s} v}_{H^{-1}_\rho(\R^3)}&\geq \norm{P_{A,-,s}\left\langle D_x,\rho\right\rangle w}_{H^{-1}_\rho(\R^3)}-\norm{P_{A,-,s}\left\langle D_x,\rho\right\rangle (1-\psi_0 )\left\langle D_x,\rho\right\rangle^{-1}\psi_1 v}_{H^{-1}_\rho(\R^3)}\\
\ &\geq \norm{P_{A,-,s}\left\langle D_x,\rho\right\rangle w}_{H^{-1}_\rho(\R^3)}-C\norm{ (1-\psi_0 )\left\langle D_x,\rho\right\rangle^{-1}\psi_1 v}_{H^{2}_\rho(\R^3)}\\
\ &\geq \norm{P_{A,-,s}\left\langle D_x,\rho\right\rangle w}_{H^{-1}_\rho(\R^3)}-{C\norm{v}_{L^2(\R^{1+n})}\over\rho^2}.\end{aligned}$$
Combining these estimates with \eqref{p2g}, we deduce that \eqref{car} holds true for a sufficiently large value of  $\rho$. Then, fixing $s$, we deduce \eqref{p2a}. \end{proof}

\section{CGO solutions }
\label{sec2}
In this section we introduce a class of CGO solutions suitable for our problem stated in an unbounded domain for magnetic Schr\"odinder equations. Like in the previous section, we fix $\Omega_1=\omega\times\R$. Our goal is to build CGO solutions for the equations \eqref{eq1} extended to the cylindrical domain $\Omega_1$ in order to consider their restrictions on $\Omega$ for proving Theorem \ref{t1}, since according to \eqref{closed} we have $\Omega\subset\Omega_1$. 

We consider  CGO solutions on $\Omega_1$ corresponding to some specific solutions $u_j\in H^1(\Omega_1)$, $j=1,2$, of
$\Delta_{A_1} u_1+q_1u_1=0$, $\Delta_{A_2} u_2+\overline{q_2}u_2=0$  in $\Omega_1$ for $A_j\in L^\infty(\Omega_1)^3\cap L^2(\Omega_1)^3$ and $q_j\in L^\infty(\Omega_1;\mathbb C)$. More precisely, like in \cite{Ki4}, we start by considering $\theta\in\mathbb S^{1}:=\{y\in\R^2:\ |y|=1\}$, $\xi'\in\theta^\bot\setminus\{0\}$ with $\theta^\bot:=\{y\in\R^2:\ y\cdot\theta=0\}$, $\xi:=(\xi',\xi_3)\in \R^3$ with  $\xi_3\neq0$. Then, we define $\eta\in\mathbb S^2:=\{y\in\R^3:\ |y|=1\}$  by
$$\eta=\frac{(\xi',-\frac{|\xi'|^2}{\xi_3})}{\sqrt{|\xi'|^2+\frac{|\xi'|^4}{\xi_3^2}}}.$$ 
It is clear that
\bel{orth}\eta\cdot\xi=(\theta,0)\cdot\xi=(\theta,0)\cdot\eta=0.\ee
We set also $\psi\in\mathcal C^\infty_0(\R;[0,1])$ such that $\psi=1$ on a neighborhood of $0$ in $\R$ and, for $\rho>1$, we consider  solutions $u_j\in H^1(\Omega_1)$ of
 $\Delta_{A_1} u_1+q_1u_1=0$, $\Delta_{A_2} u_2+\overline{q_2}u_2=0$ in $\Omega_1$ taking the form
\bel{CGO1}u_1(x',x_3)=e^{\rho \theta\cdot x'}\left(\psi\left(\rho^{-\frac{1}{4}}x_3\right)b_{1,\rho}e^{i\rho x\cdot\eta-i\xi\cdot x}+w_{1,\rho}(x',x_3)\right),\quad x'\in\omega,\ x_3\in\R,\ee
\bel{CGO2}u_2(x',x_3)=e^{-\rho \theta\cdot x'}\left(\psi\left(\rho^{-\frac{1}{4}}x_3\right)b_{2,\rho}e^{i\rho x\cdot\eta}+w_{2,\rho}(x',x_3)\right),\quad x'\in\omega,\ x_3\in\R.\ee
Here $b_{j,\rho}\in \mathcal C^\infty(\overline{\Omega_1})$ and the remainder term $w_{j,\rho}\in H^1(\Omega_1)$ satisfies the decay property 
\bel{RCGO} \lim_{\rho\to+\infty}(\rho^{-1}\norm{w_{j,\rho}}_{H^1(\Omega_1)}+\norm{w_{j,\rho}}_{L^2(\Omega_1)})=0.\ee
 This construction can be summarized in the following way.
\begin{theorem}\label{t4} For $j=1,2$ and  for all $\rho>\rho_2$, with $\rho_2$ the constant of Proposition \ref{p2}, the equations  $\Delta_{A_1} u_1+q_1u_1=0$, $\Delta_{A_2} u_2+\overline{q_2}u_2=0$,  admit respectively a solution $u_j\in H^1(\Omega_1)$ of the form \eqref{CGO1}-\eqref{CGO2} with $w_{j,\rho}$ satisfying the decay property \eqref{RCGO}.\end{theorem}
\begin{rem} \emph{Like in \cite{Ki4}, we can not  consider  CGO solutions similar to those on bounded domains since they will not be square integrable in $\Omega_1$. In a similar way to \cite{Ki4}, we consider this new expression of the CGO solutions with  principal parts that propagates in some suitable way along the axis of $\Omega_1$ with respect to the large parameter $\rho$. Comparing to \cite{Ki4} we need also to consider here  the presence of non-compactly supported magnetic potentials. This part of our construction, will be precised in the next subsection.}\end{rem} 
In order to consider suitable solutions taking the form \eqref{CGO1}-\eqref{CGO2}, we need to define first the expressions $b_{j,\rho}$ in the principal part, which will be solutions of some $\overline{\partial}$ type equation involving the magnetic potential $A_j$. Then, we will consider the remainder terms by using the Carleman estimates of the preceding section.
\subsection{Principal parts of the CGO}
In this subsection we will introduce the form of the principal part $b_{j,\rho}$, $j=1,2$, of our CGO solutions given by \eqref{CGO1}-\eqref{CGO2}. For this purpose, we assume that $b_{j,\rho}$, $j=1,2$, is an approximation of a solution $b_j$ of the  equations
\bel{trans}2(\tilde{\theta}+i\eta)\cdot \nabla b_1+2i[(\tilde{\theta}+i\eta)\cdot A_1(x)]b_1=0,\quad 2(-\tilde{\theta}+i\eta)\cdot \nabla b_2+2i[(-\tilde{\theta}+i\eta)\cdot A_2(x)]b_2=0,\quad x\in\Omega_1,\ee
here $\tilde{\theta}:=(\theta,0)\in\mathbb S^2$. This approach, also considered in \cite{BKS1,Ki4,KU,Sa1},  makes it possible to reduce the regularity assumption on the first order coefficients $A_j$. Indeed, by replacing  the functions $b_1$, $b_2$, whose regularity depends on the one  of the coefficients $A_1$ and $A_2$, with their approximation $b_{1,\rho}$, $b_{2,\rho}$,  we can weaken the regularity  assumption imposed on the  coefficients $A_j$, $j=1,2$, from $W^{2,\infty}(\Omega_1)^3$ to $ L^\infty(\Omega_1)^3$. Moreover, this approach requires also no information about the domain $\Omega$ and the  coefficients $A_j$, $j=1,2$,  on $\pd\Omega$. More precisely, if in our construction  we use the expression $b_j$ instead of  $b_{j,\rho}$, $j=1,2$,  then, following our strategy, we can prove Theorem  \ref{t1}  only for  specific  domains and for coefficients $A_1,A_2\in W^{2,\infty}(\Omega)^3\cap L^1(\Omega)$  satisfying
$$\partial_x^\alpha A_1(x)=\partial_x^\alpha A_2(x),\quad x\in\pd\Omega,\ \alpha\in\mathbb N^3,\ |\alpha|\leq1,$$
where in our case we make no assumption on the shape of $\Omega$ (except the condition $\Omega\subset\omega\times\R$) and about $A_j$ at $\pd\Omega$.

Let us also mention that comparing to results stated on bounded domains (e.g. \cite{FKSU,KLU,KU}), the magnetic potentials $A_1$, $A_2$ can not be extended to  compactly supported functions of $\R^3$. However, we can extend them into functions of $\R^3$ supported in infinite cylinder. Combining this with the fact that $A_j\in L^2(\Omega_1)^3$, we will prove how we can build CGO solutions having properties similar to the one of \cite{KU}.

In order to define $b_{j,\rho}$, $j=1,2$, we start by introducing a suitable approximation of the coefficients $A_j$, $j=1,2$. For all $r>0$, we define $B_r:=\{x\in\R^{3}:\ |x|<r\}$ and $B_r':=\{x'\in\R^{2}:\ |x'|<r\}$. We fix $\chi\in\mathcal C^\infty_0(\R^{3})$ such that $\chi\geq0$, $\int_{\R^{3}}\chi(x)dx=1$, supp$(\chi)\subset B_1$, and we define $\chi_\rho$ by
$\chi_\rho(x)=\rho^{{3\over 4}}\chi(\rho^{{1\over4}}x)$. Then, for $j=1,2$, we fix
$$A_{j,\rho}(x):=\int_{\R^{3}}\chi_\rho(x-y)A_j(y)dy.$$
Here, we assume that, for $j=1,2$, $A_j=0$ on $\R^{3}\setminus \Omega_1$.
For $j=1,2$, since $A_j\in L^2(\R^{3})^3$, by density one can check that
\bel{a1a}\lim_{\rho\to+\infty}\norm{A_{j,\rho}-A_j}_{L^2(\R^{3})}=0,\ee
and, using the fact that $A_j\in L^\infty(\R^{3})^3$, we deduce the estimates
\bel{a1b}\norm{A_{j,\rho}}_{H^k(\R^{3})}+\norm{A_{j,\rho}}_{W^{k,\infty}(\R^{3})}\leq C_k\rho ^{{k\over 4}},\ee
with $C_k$ independent of $\rho$. We remark that $$A_{\rho}(x):=\int_{\R^{3}}\chi_\rho(x-y)A(y)dy=A_{1,\rho}(x)-A_{2,\rho}(x),$$
with $A=A_1-A_2$. Recall that, for $j=1,2$, supp$(A_{j,\rho})\subset \overline{\Omega_1}+B_1:=\{x+y:\ x\in \overline{\Omega_1}, y\in B_1\}$. Moreover, fixing  $R:=\underset{x'\in\overline{\omega}}{\sup}|x'|$, $R_1:=2\sqrt{2}(R+2+\frac{R+2}{|\xi'|})$ and assuming that $|(s_1,s_2)|\geq R_1$, we find $|s_1|\geq\frac{R_1}{\sqrt{2}}$ or $|s_2|\geq\frac{R_1}{\sqrt{2}}$. In addition, since $\theta\cdot\xi'=0$, we get
$$|(s_1,s_2)|\geq R_1\Longrightarrow |s_1\theta+s_2\xi'|=|(s_1,s_2|\xi'|)|\geq \max(|s_1|,|s_2||\xi'|)>2R+4$$
and, for all $x=(x',x_3)\in B_{R+1}'\times\R$, we get
$$|(s_1,s_2)|\geq R_1\Longrightarrow |x'-s_1\theta-s_2\xi'|\geq |s_1\theta+s_2\xi'|- |x'|\geq R+3.$$
Thus, for all $x=(x',x_3)\in B_{R+1}'\times\R$, the function
$$(s_1,s_2)\mapsto A_{j,\rho}(s_1\tilde{\theta}+s_2\eta+x)$$
will be supported in $B'_{R_1}$. Thus,  we can define
\bel{conde1} \begin{aligned}&\Phi_{1,\rho}(x):=\frac{-i}{2\pi} \int_{\R^2} \frac{(\tilde{\theta}+i\eta)\cdot A_{1,\rho}(x-s_1\tilde{\theta}-s_2\eta)}{s_1+is_2}ds_1ds_2,\\
 &\Phi_{2,\rho}(x):=\frac{-i}{2\pi} \int_{\R^2} \frac{(-\tilde{\theta}+i\eta)\cdot A_{2,\rho}(x+s_1\tilde{\theta}-s_2\eta)}{s_1+is_2}ds_1ds_2.\end{aligned}\ee
Fixing
\bel{conde2} b_{1,\rho}(x)=e^{\Phi_{1,\rho}(x)},\quad b_{2,\rho}(x)=e^{\Phi_{2,\rho}(x)},\ee
we obtain 
\bel{tt}(\tilde{\theta}+i\eta)\cdot \nabla b_{1,\rho}+i[(\tilde{\theta}+i\eta)\cdot A_{1,\rho}(x)]b_{1,\rho}=0,\quad (-\tilde{\theta}+i\eta)\cdot \nabla b_{2,\rho}+i[(-\tilde{\theta}+i\eta)\cdot A_{2,\rho}(x)]b_{2,\rho}=0,\quad x\in\Omega_1.\ee
Here, even if $A_{j,\rho}$, $j=1,2$, is not compactly supported, one can use the fact that the functions
$$(s_1,s_2)\mapsto A_{j,\rho}(s_1\tilde{\theta}+s_2\eta+s_3\xi),\quad s_3\in\R,$$
are compactly supported to deduce \eqref{tt}. Moreover, using the fact that
$$(x-s_1\tilde{\theta}-s_2\eta)\notin\textrm{supp}(A_{j,\rho}),\quad x\in B'_{R+1}\times\R,\ |(s_1,s_2)|>R_1,\ j=1,2,$$
 for all $x\in B_{R+1}'\times\R,\ j=1,2,$ we deduce that
$$\begin{aligned}|\Phi_{j,\rho}(x)|&\leq \frac{1}{2\pi} \int_{|(s_1,s_2)|\leq R_1} \frac{|A_{j,\rho}(x-s_1\tilde{\theta}-s_2\eta)|}{|s_1+is_2|}ds_1ds_2\\
\ &\leq \frac{\norm{A_{j,\rho}}_{L^\infty(\R^3)}}{2\pi} \int_{|(s_1,s_2)|\leq R_1} \frac{1}{|(s_1,s_2)|}ds_1ds_2\\
\ &\leq C,\end{aligned}$$
with $C$ independent of $\rho$. This proves that
$$\norm{\Phi_{j,\rho}}_{L^\infty(B_{R+1}'\times\R)}\leq C.$$
In the same way, we can prove that
\bel{tt1}\norm{\Phi_{j,\rho}}_{W^{k,\infty}(B_{R+1}'\times\R)}\leq C_k\rho^{\frac{k}{4}},\quad k\geq0,\ee
with $C_k$ independent of $\rho$. According to this estimate, we have
\bel{cond31}\norm{b_{j,\rho}}_{W^{k,\infty}( B_{R+1}'\times\R)}\leq C_k\rho^{{k\over4}},\quad k\geq0.\ee
Moreover, conditions  \eqref{tt}, \eqref{cond31} and the fact that 
$$[\textrm{supp}(A_j)\cup \textrm{supp}(A_{j,\rho})]\subset \overline{\Omega_1}+B_1\subset B'_{R+1}\times\R,\quad j=1,2,$$
 imply that
\bel{cond5} \begin{aligned}\norm{(\tilde{\theta}+i\eta)\cdot \nabla b_{1,\rho}+i[(\tilde{\theta}+i\eta)\cdot A_{1}]b_{1,\rho}}_{L^2(B_{R+1}'\times\R)}&=\norm{[i[(\tilde{\theta}+i\eta)\cdot (A_{1}-A_{1,\rho})]]b_{1,\rho}}_{L^2(B_{R+1}'\times\R)}\\
\ &\leq C\norm{A_{1}-A_{1,\rho}}_{L^2(\R^3)},\end{aligned}\ee
\bel{cond6} \begin{aligned}\norm{(-\tilde{\theta}+i\eta)\cdot \nabla b_{2,\rho}+i[(-\tilde{\theta}+i\eta)\cdot A_2]b_{2,\rho}}_{L^2(B_{R+1}'\times\R)}&=\norm{[i[(\tilde{\theta}+i\eta)\cdot (A_{2}-A_{2,\rho})]]b_{2,\rho}}_{L^2(B_{R+1}'\times\R)}\\
\ &\leq C\norm{A_{2}-A_{2,\rho}}_{L^2(\R^3)},\end{aligned}\ee
 with $C>0$ independent of $\rho$.
Using these properties of the expressions $b_{j,\rho}$, $j=1,2$, we will complete the construction of the solutions $u_j$ of the form \eqref{CGO1}-\eqref{CGO2}.

\subsection{Remainder term of the CGO solutions}
In this subsection we will construct the remainder term $w_{j,\rho}$, $j=1,2$, appearing in \eqref{CGO1}-\eqref{CGO2} and satisfying the decay property \eqref{RCGO}. For this purpose, we will combine the Carleman estimate \eqref{p2a} with the properties of the expressions $b_{j,\rho}$, $j=1,2$, in order to complete the construction of these solutions. In this subsection, we assume that $\rho>\rho_2$ with $\rho_2$ the constant introduced in Proposition \ref{p2}. The proof for the existence of the remainder term $w_{1,\rho}$ and $w_{2,\rho}$ being similar, we will only show the existence of $w_{1,\rho}$. Let us first remark that $w_{1,\rho}$ should be a solution of the equation
\bel{t4a} P_{A_1,q_1,+}w=e^{-\rho \theta\cdot x'}(\Delta_{A_1}+q_1)e^{\rho \theta\cdot x'}w=e^{i\rho \eta\cdot x}F_{1,\rho}(x),\quad x\in\Omega_1,\ee
with $F_{1,\rho}$ defined, for all $x=(x',x_3)\in B'_{R+1}\times\R$ (we recall that $B_r'=\{x'\in\R^{2}:\ |x'|<r\}$ and  
$R=\underset{x'\in\overline{\omega}}{\sup}|x'|$), by
\bel{t4b}\begin{aligned}F_{1,\rho}(x)&=-e^{-\rho \theta\cdot x'-i\rho \eta\cdot x}(\Delta_{A_1}+q_1)\left[e^{\rho \theta\cdot x'+i\rho \eta\cdot x}\psi\left(\rho^{-\frac{1}{4}}x_3\right)b_{1,\rho}e^{-i\xi\cdot x}\right]\\
&=-\left((-|\xi|^2+\textrm{div}(A_1)+q_1)\psi\left(\rho^{-\frac{1}{4}}x_3\right)+2i\eta_3\rho^{\frac{3}{4}}\psi'\left(\rho^{-\frac{1}{4}}x_3\right)-2i\xi_3\rho^{-\frac{1}{4}}\psi'\left(\rho^{-\frac{1}{4}}x_3\right)\right)b_{1,\rho}e^{-i\xi\cdot x}\\
&\ \ \  -\left[\rho^{-\frac{1}{2}}\psi''\left(\rho^{-\frac{1}{4}}x_3\right)b_{1,\rho}+ 2\pd_{x_3}b_{1,\rho}\rho^{-\frac{1}{4}}\psi'\left(\rho^{-\frac{1}{4}}x_3\right)-i2\xi\cdot\nabla b_{1,\rho}\psi\left(\rho^{-\frac{1}{4}}x_3\right)\right]e^{-i\xi\cdot x}\\
&\ \ \ -2\rho[(\tilde{\theta}+i\eta)\cdot \nabla b_{1,\rho}+i[(\tilde{\theta}+i\eta)\cdot A_{1}]b_{1,\rho}]\psi\left(\rho^{-\frac{1}{4}}x_3\right)e^{-i\xi\cdot x}.\end{aligned}\ee
Here we consider $A_1$ as an element of $L^\infty(\R^3)^3\cap L^2(\R^3)^3$ satisfying $A_1=0$ on $\R^3\setminus\Omega_1$. We fix $\phi\in\mathcal C^\infty_0(B'_{R+1};[0,1])$ satisfying $\phi=1$ on $B'_{R+\frac{1}{2}}$, and we define $$G_\rho(x',x_3):=\phi(x')F_{1,\rho}(x',x_3),\quad x'\in\R^2,\ x_3\in\R,$$
$$K_\rho(x):=G_\rho(x)-\phi(x')\psi\left(\rho^{-\frac{1}{4}}x_3\right)\textrm{div}(A_1)b_{1,\rho}e^{-i\xi\cdot x} ,\quad x'\in\R^2,\ x_3\in\R,\ x=(x',x_3).$$
It is clear that $K_\rho\in L^2(\R^3)$ and in view of \eqref{cond31}-\eqref{cond6} and the fact that, using a change of variable, we find
$$\norm{\chi\left(\rho^{-\frac{1}{4}}x_3\right)}_{L^2(B'_{R+1}\times\R)}+\norm{\chi'\left(\rho^{-\frac{1}{4}}x_3\right)}_{L^2(B'_{R+1}\times\R)}+\norm{\chi''\left(\rho^{-\frac{1}{4}}x_3\right)}_{L^2(B'_{R+1}\times\R)}\leq C\rho^{\frac{1}{8}},$$
we deduce that
\bel{t4h}\norm{K_\rho}_{H^{-1}_\rho(\R^3)}\leq \rho^{-1}\norm{K_\rho}_{L^2(\R^3)}=\rho^{-1}\norm{K_\rho}_{L^2(B'_{R+1}\times\R)}\leq C(\norm{A_1-A_{1,\rho}}_{L^2(\R^3)^3}+\rho^{-\frac{1}{8}}).\ee
In the same way, since supp$(\textrm{div}(A))\subset \overline{\omega}\times\R\subset B'_{R+\frac{1}{2}}\times\R$, we have
$$\phi(x')\psi\left(\rho^{-\frac{1}{4}}x_3\right)\textrm{div}(A_1)b_{1,\rho}=\psi\left(\rho^{-\frac{1}{4}}x_3\right)\textrm{div}(A_1)b_{1,\rho}.$$
Moreover, fixing
$$c_{1,\rho}(x):=\psi\left(\rho^{-\frac{1}{4}}x_3\right)b_{1,\rho}(x),\quad x=(x',x_3)\in\R^2\times\R,$$
for any $h\in H^1_\rho(\R^3)$, we obtain
$$\begin{aligned}&\abs{\left\langle  \textrm{div}(A_1)c_{1,\rho},h\right\rangle_{H^{-1}_\rho(\R^3), H^1_\rho(\R^3)}}\\
&\leq \abs{\left\langle  A_1\cdot \nabla c_{1,\rho},h\right\rangle_{L^2(\R^3)}}+\abs{\left\langle   c_{1,\rho},A_1\cdot \nabla h\right\rangle_{L^2(\R^3)}}\\
&\leq \abs{\left\langle  A_1\cdot \nabla c_{1,\rho},h\right\rangle_{L^2(\R^3)}}+\abs{\left\langle   c_{1,\rho},(A_1-A_{1,\rho})\cdot \nabla h\right\rangle_{L^2(\R^3)}}+\abs{\left\langle   c_{1,\rho},A_{1,\rho}\cdot \nabla h\right\rangle_{L^2(\R^3)}}\\ 
&\leq \left(\norm{c_{1,\rho}}_{W^{1,\infty} (\Omega_1)}\norm{A_1}_{L^2(\Omega_1)^3}\rho^{-1}+\norm{c_{1,\rho}}_{L^\infty (B'_{R+1}\times\R)}\norm{A_1-A_{1,\rho}}_{L^2(\R^3)^3}\right)\norm{h}_{H^1_\rho(\R^3)}+\abs{\left\langle   \textrm{div}({c_{1,\rho}A_{1,\rho}}), h\right\rangle_{L^2(\R^3)}}\\
&\leq \left(2\norm{c_{1,\rho}}_{W^{1,\infty} (B'_{R+1}\times\R)}[\norm{A_1}_{L^2(\Omega_1)^3}+\norm{A_{1,\rho}}_{H^1(\R^3)^3}]\rho^{-1}+\norm{c_{1,\rho}}_{L^\infty (B'_{R+1}\times\R)}\norm{A_1-A_{1,\rho}}_{L^2(\R^3)^3}\right)\norm{h}_{H^1_\rho(\R^3)}.\end{aligned}$$
Here we use the fact that supp$(A_{1,\rho})\subset \Omega_1+B_1\subset B'_{R+1}\times\R$. Combining this with \eqref{a1b} and \eqref{cond31}, we find
$$\abs{\left\langle  \textrm{div}(A_1)c_{1,\rho},h\right\rangle_{H^{-1}_\rho(\R^3), H^1_\rho(\R^3)}}\leq C(\rho^{-\frac{3}{4}}+\norm{A_1-A_{1,\rho}}_{L^2(\R^3)^3})\norm{h}_{H^1_\rho(\R^3)}$$
and it follows
$$\norm{\psi\left(\rho^{-\frac{1}{4}}x_3\right)\textrm{div}(A_1)b_{1,\rho}}_{H^{-1}_\rho(\R^3)}\leq C(\rho^{-\frac{3}{4}}+\norm{A_1-A_{1,\rho}}_{L^2(\R^3)^3}).$$
Then, \eqref{t4h} implies
\bel{t4c} \norm{G_\rho}_{H^{-1}_\rho(\R^3)}\leq C(\norm{A_{1}-A_{1,\rho}}_{L^2(\R^3)^3}+\rho^{-\frac{1}{8}}).\ee
From now on, combining \eqref{p2a} with \eqref{t4c}, we will complete the construction of the remainder term $w_{1,\rho}$ by using a classical duality argument. More precisely, applying \eqref{p2a}, we consider the linear form 
$T_\rho$ defined on $\mathcal Q:=\{P_{A_1,\overline{q_1},-}w: w\in\mathcal C^\infty_0(\Omega_1)\}$ by
$$T_\rho(P_{A_1,\overline{q_1},-}v):=\overline{\left\langle G_{\rho}, e^{-i\rho \eta\cdot x}v\right\rangle_{H^{-1}_\rho(\R^3), H^1_\rho(\R^3)}},\quad v\in\mathcal C^\infty_0(\Omega_1).$$
Here and from now on we define the duality bracket $\left\langle \cdot,\cdot\right\rangle_{H^{-1}_\rho(\R^3), H^1_\rho(\R^3)}$ in the complex sense, which means that 
$$\left\langle v,w\right\rangle_{H^{-1}_\rho(\R^3), H^1_\rho(\R^3)}=\left\langle v,w\right\rangle_{L^2(\R^3)}=\int_{\R^3}v\overline{w}dx,\quad v\in L^2(\R^3),\ w\in H^1(\R^3).$$
Applying again \eqref{p2a}, for all $v\in\mathcal C^\infty_0(\Omega_1)$, we obtain
$$\begin{aligned}|T_\rho(P_{A_1,\overline{q_1},-}v)|&\leq \norm{G_{\rho}}_{H^{-1}_\rho(\R^3)}\norm{e^{-i\rho \eta\cdot x}v}_{H^1_\rho(\R^3)}\\
\ &\leq 2\rho\norm{G_{\rho}}_{H^{-1}_\rho(\R^3)}\rho^{-1}\norm{v}_{H^1_\rho(\R^3)}\\
\ &\leq C\rho\norm{G_{\rho}}_{H^{-1}_\rho(\R^3)}\norm{P_{A_1,\overline{q_1},-}v}_{H^{-1}_\rho(\R^3)},\end{aligned}$$
with $C>0$ independent of $\rho$. Thus, applying the Hahn-Banach theorem, we deduce that $T_\rho$ admits an extension as a continuous linear form on ${H^{-1}_\rho(\R^3)}$ whose norm will be upper bounded by $C\rho\norm{G_{\rho}}_{H^{-1}_\rho(\R^3)}$. Therefore, there exists $w_{1,\rho}\in H^1_\rho(\R^3)$ such that
\bel{t4d}\left\langle P_{A_1,\overline{q_1},-}v, w_{1,\rho}\right\rangle_{H^{-1}_\rho(\R^3), H^1_\rho(\R^3)}=T_\rho(P_{A_1,\overline{q_1},-}v)=\overline{\left\langle G_{\rho}, e^{-i\rho \eta\cdot x}v\right\rangle_{H^{-1}_\rho(\R^3), H^1_\rho(\R^3)}},\quad v\in\mathcal C^\infty_0(\Omega_1),\ee
\bel{t4e}\norm{w_{1,\rho}}_{H^1_\rho(\R^3)}\leq C\rho\norm{G_{\rho}}_{H^{-1}_\rho(\R^3)}.\ee
From \eqref{t4d} and the fact that, for all $x\in\Omega_1$, $G_\rho(x)=F_{1,\rho}(x)$, we obtain
$$\begin{aligned}\left\langle  P_{A_1,q_1,+}w_{1,\rho},v\right\rangle_{D'(\Omega_1), \mathcal C^\infty_0(\Omega_1)}&=\overline{\left\langle P_{A_1,\overline{q_1},-}v, w_{1,\rho}\right\rangle_{H^{-1}_\rho(\R^3), H^1_\rho(\R^3)}}\\
\ &=\left\langle G_{\rho}, e^{-i\rho \eta\cdot x}v \right\rangle_{H^{-1}_\rho(\R^3), H^1_\rho(\R^3)}\\
\ &=\left\langle e^{i\rho \eta\cdot x}F_{1,\rho}, v \right\rangle_{D'(\Omega_1), \mathcal C^\infty_0(\Omega_1)}.\end{aligned}$$
It follows that $w_{1,\rho}$ solves 
$P_{A_1,q_1,+}w_{1,\rho}=e^{i\rho \eta\cdot x}F_{1,\rho}$ in $\Omega_1$ and $u_1$ given by \eqref{CGO1} is a solution of $\Delta_{A_1}u+q_1u=0$ in $\Omega_1$ lying in $H^1(\Omega_1)$. In addition, from \eqref{t4c} and \eqref{t4e}, we deduce that
\bel{tutu}\rho^{-1}\norm{w_{1,\rho}}_{H^1(\Omega_1)}+\norm{w_{1,\rho}}_{L^2(\Omega_1)}\leq 2\rho^{-1}\norm{w_{1,\rho}}_{H^1_\rho(\R^3)}\leq C(\norm{A_{1}-A_{1,\rho}}_{L^2(\R^3)^3}+\rho^{-\frac{1}{8}})\ee
which implies the decay property \eqref{RCGO}. This completes the proof of Theorem \ref{t4}.

\section{Uniqueness result}
In this section we will use the result of the preceding section in order to complete the proof of Theorem \ref{t1}. Namely under the assumption of Theorem \ref{t1}, we will show that \eqref{t1a} implies that $dA_1=dA_2$. Then, assuming $A=A_1-A_2\in \mathcal C(\R^3)$, we will prove that  $q_1=q_2$. For $j=1,2$, we assume that $A_j\in L^\infty(\R^3)^3\cap L^2(\R^3)^3$ and $q_j\in L^\infty(\R^3;\mathbb C)$ with $A_j$ and $q_j$ extended by $0$ on $\R^3\setminus \Omega$. We use here the notation of the previous sections and we assume that $A=A_1-A_2\in L^1(\R^3)$. We start with the recovery of the magnetic field.

\subsection{Recovery of the magnetic field}
In this subsection we will prove that \eqref{t1a} implies that $dA_1=dA_2$.
Let us first remark that $A_\rho=A_{1,\rho}-A_{2,\rho}=\chi_\rho*A$ and, since $A\in L^1(\R^3)^3$, by density one can check that
\bel{t1c}\lim_{\rho\to+\infty}\norm{A_\rho-A}_{L^1(\R^3)}=0.\ee
For $j=1,2$, we fix $u_j\in H^{1}(\Omega_1)$  a solution of  $\Delta_{A_1} u_1+q_1u_1=0$, $\Delta_{A_2} u_2+\overline{q_2}u_2=0$ in $\Omega_1$ of the form \eqref{CGO1}-\eqref{CGO2} with $\rho>\rho_2$ and with $w_{j,\rho}$ satisfying \eqref{RCGO}. In view of \eqref{closed}, we can see that the restriction of $u_1$ (resp. $u_2$) to $\Omega$ is lying in $H^1(\Omega)$ and it solves the equation $\Delta_{A_1} u_1+q_1u_1=0$ (resp. $\Delta_{A_2} u_2+\overline{q_2}u_2=0$) in $\Omega$. From now on, we consider the restriction to $\Omega$ of these CGO solutions initially defined on $\Omega_1$.

In view of \eqref{t1a}, we can find $v_2\in H^1(\Omega)$ satisfying $\Delta_{A_2} v_2+q_2v_2=0$ with $\tau v_2=\tau u_1$ and $N_{A_1,q_1} u_1=N_{A_2,q_2} v_2$. Therefore, we have
$$\begin{aligned}0=\left\langle N_{A_1,q_1}u_1,\tau u_2\right\rangle-\left\langle N_{A_2,q_2}v_2,\tau u_2\right\rangle&=\left\langle N_{A_1,q_1}u_1,\tau u_2\right\rangle-\overline{\left\langle N_{A_2,\overline{q_2}}u_2,\tau v_2\right\rangle}\\
\ &=\left\langle N_{A_1,q_1}u_1,\tau u_2\right\rangle-\overline{\left\langle N_{A_2,\overline{q_2}}u_2,\tau u_1\right\rangle}\\
\ &=i\int_{\R^3}(A\cdot\nabla u_1)\overline{u_2}dx -i\int_{\R^3}u_1(\overline{A\cdot\nabla u_2})dx+\int_{\R^3}\tilde{q}u_1\overline{u_2}dx,\end{aligned}$$
where $\tilde{q}=|A_2|^2-|A_1|^2+q$, with $q=q_1-q_2$ extended by zero to $\R^3$. According to \eqref{RCGO}, \eqref{cond31} and the fact that $A\in L^1(\R^3)$, multiplying  this expression by $-i\rho^{-1}2^{-1}$ and sending $\rho\to+\infty$, we find
$$\begin{aligned}&\lim_{\rho\to+\infty}\int_{\R^3}(A\cdot(\tilde{\theta}+i\eta))\exp\left(\Phi_{1,\rho}+\overline{\Phi_{2,\rho}}\right)e^{-ix\cdot\xi}dx\\
&=\lim_{\rho\to+\infty}\int_{\R^3}\psi^2(\rho^{-\frac{1}{4}}x_3)(A\cdot(\tilde{\theta}+i\eta))\exp\left(\Phi_{1,\rho}+\overline{\Phi_{2,\rho}}\right)e^{-ix\cdot\xi}dx=0.\end{aligned}$$
Here we use \eqref{tt1} and the fact that by Lebesgue dominate convergence theorem
$$\lim_{\rho\to+\infty}\norm{A-\psi^2(\rho^{-\frac{1}{4}}x_3)A}_{L^1(\R^3)}=0.$$
Combining this with \eqref{tt1} and \eqref{t1c}, we obtain
$$\lim_{\rho\to+\infty}\int_{\R^3}(A_\rho\cdot(\tilde{\theta}+i\eta))\exp\left(\Phi_{1,\rho}+\overline{\Phi_{2,\rho}}\right)e^{-ix\cdot\xi}dx=0.$$
On the other hand, one can easily check that
$$\Phi_\rho=\Phi_{1,\rho}+\overline{\Phi_{2,\rho}}=\frac{-i}{2\pi}\int_{\R^2}\frac{(\tilde{\theta}+i\eta)\cdot A_{\rho}(x-s_1\tilde{\theta}-s_2\eta)}{s_1+is_2}ds_1ds_2.$$
and we deduce that
\bel{t1d}\lim_{\rho\to+\infty}\int_{\R^3}(A_\rho\cdot(\tilde{\theta}+i\eta))e^{\Phi_{\rho}}e^{-ix\cdot\xi}dx=0.\ee
Now let us consider the following intermediate result.
\begin{lemma} \label{l3} We have
\bel{l3a} \int_{\R^3}(A_\rho\cdot(\tilde{\theta}+i\eta))e^{\Phi_{\rho}}e^{-ix\cdot\xi}dx=(\tilde{\theta}+i\eta)\cdot\left(\int_{\R^3}A_\rho(x)e^{-ix\cdot\xi}dx\right)=(2\pi)^{\frac{3}{2}}(\tilde{\theta}+i\eta)\cdot\mathcal F(A_\rho)(\xi).\ee
\end{lemma}
\begin{proof} For $A_\rho$ compactly supported this result is well known and one can refer to \cite[Proposition 3.3]{KU} or \cite[Lemma 6.2]{Sa2} for its proof. Since here we deal with non-compactly supported magnetic potentials, the proof of the result will be required. From now on, to every $x\in\R^3$, we associate the coordinate $(x'',x_*)\in\R^2\times \R$, with $x''=(x'_1,x'_2)=(x\cdot\tilde{\theta},x\cdot\eta)$ and $x_*=\frac{x\cdot\xi}{|\xi|}$.
Recall that supp$(A_\rho)\subset B'_{R+1}\times\R$ and, fixing $\tilde{A}_\rho: (x'',x_*)\mapsto A_\rho(x)$, in a similar way to Subsection 3.1, we find 
$$\textrm{supp}(\tilde{A}_\rho)\subset (-R-1,R+1)\times\left(-\frac{(R+1)}{|\xi'|},\frac{R+1}{|\xi'|}\right)\times \R\subset B'_{R_1}\times\R.$$
 Thus, fixing $\tilde{\Phi}_\rho: (x'',x_*)\mapsto \Phi_\rho(x)$, for $|x''|>R_1$ we have
$$\tilde{\Phi}_\rho(x'',x_*)=\frac{-i}{2\pi}\int_{B'_{R_1}}\frac{(\tilde{\theta}+i\eta)\cdot\tilde{A}_\rho(y'',x_*)}{x'_1-y'_1+i(x'_2-y'_2)}dy''.$$
It follows that
$$|\tilde{\Phi}_\rho(x'',x_*)|\leq \frac{\norm{A_\rho}_{L^\infty(\R^3)}|B'_{R_1}|}{2\pi(|x''|-R_1)},\quad |x''|>R_1,\ x_*\in\R$$
and in particular, for every $x_*\in\R$, we get
\bel{l3b}|\tilde{\Phi}_\rho(x'',x_*)|=\underset{|x''|\to+\infty}{\mathcal O}\left(|x''|^{-1}\right).\ee
On the other hand, using the fact that
$$(\pd_{x'_1}+i\pd_{x'_2})\tilde{\Phi}_\rho(x'',x_*)=(\tilde{\theta}+i\eta)\nabla \Phi_\rho=-iA_\rho\cdot (\tilde{\theta}+i\eta)$$
and the fact that $A_\rho\in L^1(\R^3)$, by Fubini's theorem
we find
\bel{l3c}\int_{\R^3}(A_\rho\cdot(\tilde{\theta}+i\eta))e^{\Phi_{\rho}}e^{-ix\cdot\xi}dx=i\int_{\R}\left(\int_{\R^2}(\pd_{x'_1}+i\pd_{x'_2})e^{\tilde{\Phi}_\rho(x'',x_*)}dx''\right)e^{-ix_*|\xi|}dx_*.\ee
Moreover, for all $r>0$ fixing $n=(n_1,n_2)$ the outward unit normal vector to $B'_r$, we  have
$$\int_{|x''|<r}(\pd_{x'_1}+i\pd_{x'_2})e^{\tilde{\Phi}_\rho(x'',x_*)}dx''=\int_{|x''|=r}e^{\tilde{\Phi}_\rho(x'',x_*)}(n_1+in_2)d\sigma(x'').$$
Applying \eqref{l3b}, we find
$$e^{\tilde{\Phi}_\rho(x'',x_*)}=1+\tilde{\Phi}_\rho(x'',x_*)+\underset{|x''|\to+\infty}{\mathcal O}\left(|x''|^{-2}\right)$$
and it follows
\bel{l3d}\int_{|x''|<r}(\pd_{x'_1}+i\pd_{x'_2})e^{\tilde{\Phi}_\rho(x'',x_*)}dx''= \int_{|x''|=r}(n_1+in_2)d\sigma(x'')+\int_{|x''|=r}\tilde{\Phi}_\rho(x'',x_*)(n_1+in_2)d\sigma(x'')+\underset{r\to+\infty}{\mathcal O}\left(r^{-1}\right).\ee
In addition, we get
$$\int_{|x''|=r}(n_1+in_2)d\sigma(x'')=\int_{|x''|<r}(\pd_{x'_1}+i\pd_{x'_2})1dx''=0,$$
$$\int_{|x''|=r}\tilde{\Phi}_\rho(x'',x_*)(n_1+in_2)d\sigma(x'')=\int_{|x''|<r}(\pd_{x'_1}+i\pd_{x'_2})\tilde{\Phi}_\rho(x'',x_*)dx''$$
and sending $r\to+\infty$ in \eqref{l3d}, we obtain
$$\begin{aligned}\int_{\R^3}(A_\rho\cdot(\tilde{\theta}+i\eta))e^{\Phi_{\rho}}e^{-ix\cdot\xi}dx&=i\int_{\R}\left(\int_{\R^2}(\pd_{x'_1}+i\pd_{x'_2})\tilde{\Phi}_\rho(x'',x_*)dx''\right)e^{-ix_*|\xi|}dx_*\\
\ &=\int_{\R}\left(\int_{\R^2}(\tilde{\theta}+i\eta)\cdot\tilde{A}_\rho(x'',x_*)dx''\right)e^{-ix_*|\xi|}dx_*.\end{aligned}$$
From this identity, we deduce \eqref{l3a}.
\end{proof}
Combining \eqref{t1c} and \eqref{t1d}-\eqref{l3a}, we obtain
$$(\tilde{\theta}+i\eta)\cdot\mathcal F(A)(\xi)=\lim_{\rho\to+\infty}(\tilde{\theta}+i\eta)\cdot\mathcal F(A_\rho)(\xi)=0.$$
In the same way, replacing $\eta$ by $-\eta$ in our analysis, we find $(\tilde{\theta}-i\eta)\cdot\mathcal F(A)(\xi)=0$ and it follows $\tilde{\theta}\cdot\mathcal F(A)(\xi)=\eta\cdot\mathcal F(A)(\xi)=0$. Combining this with the fact that $(\tilde{\theta},\eta)$ is an orthonormal basis of $\xi^\bot=\{y\in\R^3:\ y\cdot\xi=0\}$, we find
\bel{t1e}\zeta \cdot\mathcal F(A)(\xi)=0,\quad \zeta\in\xi^\bot.\ee
Moreover, for $1\leq j<k\leq3$, fixing $\zeta=\xi_ke_j-\xi_j e_k,$
with $$e_j=(0,\ldots,0,\underbrace{1}_{\textrm{position\ } j},0,\ldots0),\quad e_k=(0,\ldots,0,\underbrace{1}_{\textrm{position } k},0,\ldots0),$$ \eqref{t1e} implies
\bel{t1f}\xi_k \mathcal F(a_j)(\xi)-\xi_j \mathcal F(a_k)(\xi)=0,\quad 1\leq j<k\leq3,\ee
where $A=(a_1,a_2,a_3)$. Recall that so far, we have proved \eqref{t1f} for any $\xi=(\xi',\xi)\in\R^2\times\R$ with $\xi'\neq0$ and $\xi_3\neq0$. Since $A\in L^1(\R^3)^3$ we can extend this identity to 	any $\xi\in\R^3$ by using the continuity of $\mathcal F(A)$. Then, we deduce from \eqref{t1f} that
$$-i\mathcal F(\pd_{x_k}a_j-\pd_{x_j}a_k)(\xi)=\xi_k \mathcal F(a_j)(\xi)-\xi_j \mathcal F(a_k)(\xi)=0,\quad 1\leq j<k\leq3,\ \xi\in\R^3.$$
This proves that in the sense of distribution we have $dA=0$ and $dA_1=dA_2$. 

\subsection{Recovery of the electric potential}

In this subsection we assume that \eqref{t1a}, $A\in L^\infty(\R^3)^3$, $dA=0$ are fulfilled and we will prove that $q_1=q_2$. We start, with the following.
\begin{lemma} \label{l4} Let $A=(a_1,\ldots,a_3)\in  L^\infty(\R^3)^3$. Assume that  $dA=0$,  and  fix
\bel{l4aa}\phi(x):=\int_0^1A(sx)\cdot xds,\quad x\in\R^3.\ee
Then, we have $\phi\in W^{1,\infty}_{loc}(\R^3)$ and $\nabla \phi=A$.
\end{lemma}
\begin{proof} Note first that since $A\in L^\infty(\R^3)^3$, we have $\phi\in L^\infty_{loc}(\R^3)$. Let $\psi\in \mathcal C^\infty_0(\R^3)$ and consider $j\in\{1,2,3\}$. We have
$$\begin{aligned}\left\langle \partial_{x_j}\phi,\psi\right\rangle_{D'(\R^3),\mathcal C^\infty_0(\R^3)}&=-\left\langle \phi,\partial_{x_j}\psi\right\rangle_{L^2(\R^3)}\\
\ &=-\sum_{k=1}^{3}\int_{\R^3}\int_0^1 x_ka_k(sx)\partial_{x_j}\psi(x)dsdx\\
\ &=-\sum_{k=1}^{3}\int_0^1 \int_{\R^3}x_ka_k(sx)\partial_{x_j}\psi(x)dxds.\end{aligned}$$
Applying the change of variable $y=sx$ and then $t=s^{-1}$, we obtain
$$\begin{aligned}\left\langle \partial_{x_j}\phi,\psi\right\rangle_{D'(\R^3),\mathcal C^\infty_0(\R^3)}&=-\sum_{j=1}^{3}\int_0^1 s^{-4}\left(\int_{\R^3}y_ja_j(y)\partial_{x_j}\psi(s^{-1}y)dy\right)ds\\
\ &=-\sum_{k=1}^{3}\int_1^{+\infty} t^{2}\int_{\R^3}y_ka_k(y)\partial_{x_j}\psi(ty)dydt\\
\ &=\int_1^{+\infty} t\left\langle \partial_{x_j} \left(\sum_{k=1}^{3}x_ka_k\right),\psi(t\cdot)\right\rangle_{D'(\R^3),\mathcal C^\infty_0(\R^3)}dt,\end{aligned}$$
with, for $\tau\in\R$, $\psi(\tau\cdot):=x\mapsto\psi(\tau x)$.
On the other hand, we have
$$\begin{aligned}&\left\langle \partial_{x_j} \left(\sum_{k=1}^{3}x_ka_k\right),\psi(t\cdot)\right\rangle_{D'(\R^3),\mathcal C^\infty_0(\R^3)}\\
&=\left\langle a_j,\psi(t\cdot)\right\rangle_{D'(\R^3),\mathcal C^\infty_0(\R^3)}+\left\langle  \left(\sum_{k=1}^{3}x_k\partial_{x_j}a_k\right),\psi(t\cdot)\right\rangle_{D'(\R^3),\mathcal C^\infty_0(\R^3)}\end{aligned}$$
and using the fact that $dA=0$, we get
$$\begin{aligned}&\left\langle \partial_{x_j} \left(\sum_{k=1}^{3}x_ka_k\right),\psi(t\cdot)\right\rangle_{D'(\R^3),\mathcal C^\infty_0(\R^3)}\\
&=\left\langle a_j,\psi(t\cdot)\right\rangle_{D'(\R^3),\mathcal C^\infty_0(\R^3)}+\left\langle  \left(\sum_{k=1}^{3}x_k\partial_{x_k}a_j\right),\psi(t\cdot)\right\rangle_{D'(\R^3),\mathcal C^\infty_0(\R^3)}\\
&=-2\left\langle a_j,\psi(t\cdot)\right\rangle_{D'(\R^3),\mathcal C^\infty_0(\R^3)}-t\left\langle a_j, \left(\sum_{k=1}^{3}x_k\partial_{x_k}\psi(t\cdot)\right)\right\rangle_{D'(\R^3),\mathcal C^\infty_0(\R^3)}. \end{aligned}$$
It follows
$$\begin{aligned}&\left\langle \partial_{x_j}\phi,\psi\right\rangle_{D'(\R^3),\mathcal C^\infty_0(\R^3)}\\
&=-\int_1^{+\infty} 2t\left\langle a_j,\psi(t\cdot)\right\rangle_{D'(\R^3),\mathcal C^\infty_0(\R^3)}dt-\int_1^{+\infty} t^2\partial_t\left\langle a_j, \psi(t\cdot)\right\rangle_{D'(\R^3),\mathcal C^\infty_0(\R^3)}dt\\
&=-\int_1^{+\infty}\partial_t\left[t^2\left\langle a_j, \psi(t\cdot)\right\rangle_{D'(\R^3),\mathcal C^\infty_0(\R^3)}\right]dt\\
&=\left\langle a_j, \psi\right\rangle_{D'(\R^3),\mathcal C^\infty_0(\R^3)}-\lim_{t\to+\infty} t^2\left\langle a_j, \psi(t\cdot)\right\rangle_{D'(\R^3),\mathcal C^\infty_0(\R^3)}=\left\langle a_j, \psi\right\rangle_{D'(\R^3),\mathcal C^\infty_0(\R^3)}.\end{aligned}$$
This proves that $\nabla_x\phi=A$ and it completes the proof of the lemma.\end{proof}
According to Lemma \ref{l4}, the function $\phi\in W^{1,\infty}_{loc}(\R^3)$ given by \eqref{l4aa} satisfies $\nabla\phi=A$. Since $\omega$ is simply connected $\Omega_1=\omega\times\R$ is also simply connected and $\R^3\setminus\Omega_1$ is connected. Therefore,  according to the fact that $A=0$ in $\R^3\setminus\Omega_1$, by extracting a constant to $\phi$ we may assume that $\phi=0$ on $\R^3\setminus\Omega_1$. Thus, we have $\phi_{|\partial\Omega_1}=0$. Note also that by eventually extending $\omega$, we may assume that $\Omega_1$ contains a neighborhood of $\overline{\Omega}$.  Now, for $A\in L^\infty(\Omega_1)^3$ and $q\in L^\infty(\Omega_1)$ let us consider the set of data
$$\mathcal D_{1,A,q}:=\{(\tau_1 u, N_{1,A,q}u):\ u\in H^1(\Omega_1),\  \Delta_Au+qu=0\},$$
where $\tau_1$ is the extension of the map $u\mapsto u_{|\pd\Omega_1}$ and, for any solution $u\in H^1(\Omega_1)$ of $\Delta_Au+qu=0$ on $\Omega_1$, $N_{1,A,q}u$ denotes the unique elements of $H^{-\frac{1}{2}}(\pd\Omega_1)$ satisfying
$$\left\langle N_{1,A,q}u,\tau_1 g\right\rangle_{H^{-\frac{1}{2}}(\pd\Omega_1),H^{\frac{1}{2}}(\pd\Omega_1), }=-\int_{\Omega_1} (\nabla+iA)u\cdot \overline{(\nabla+iA)g}dx+\int_{\Omega_1} qu\overline{g}dx,\ g\in H^1(\Omega_1).$$
Repeating some arguments of \cite[Proposition 3.4]{KU} (see also \cite[Lemma 4.2]{Sa1}), one can easily check the following.
\begin{proposition}\label{p4} For $j=1,2$, let $A_j\in L^\infty(\Omega_1)^3$, $q_j\in L^\infty(\Omega_1)$ and assume that
$$A_1(x)=A_2(x),\quad q_1(x)=q_2(x),\quad x\in\Omega_1\setminus\Omega.$$
Then the condition \eqref{t1a} implies that $\mathcal D_{1,A_1,q_1}=\mathcal D_{1,A_2,q_2}$.
\end{proposition}
In view of this result and the fact that $A_1=A_2=0$ and $q_1=q_2=0$ on $\Omega_1\setminus \Omega$, we deduce that $\mathcal D_{1,A_1,q_1}=\mathcal D_{1,A_2,q_2}$. Moreover, using the fact that $A_1-A_2=\nabla\phi$ with $\phi\in W^{1,\infty}_{loc}(\Omega_1)$ satisfying $\phi_{|\R^3\setminus\Omega_1}=0$, we obtain
$$\mathcal D_{1,A_1,q_2}=\mathcal D_{1,A_2+\nabla\phi,q_2}=\mathcal D_{1,A_2,q_2}=\mathcal D_{1,A_1,q_1}.$$
Therefore, repeating the argumentation of Section 4.1, with $A_1=A_2$, we find
\bel{t1g}\lim_{\rho\to+\infty}\int_{\R^3} q(x)\psi^2(\rho^{-\frac{1}{4}} x_3)e^{-ix\cdot \xi}dx=0,\ee
for all $\xi=(\xi',\xi_3)\in\R^2\times\R$ with $\xi'\neq0$ and $\xi_3\neq0$. Here we have used the fact that, following our definition, $A_{1,\rho}=A_{2,\rho}$, $\overline{\Phi_{2,\rho}}=-\Phi_{1,\rho}$ and $b_{1,\rho}\overline{b_{2,\rho}}=1$. In \eqref{t1g}, we can assume for instance that $\psi=1$ on $[-1,1]$. We fix $q_\rho(x',x_3)=q(x',x_3)\psi^2(\rho^{-\frac{1}{4}} x_3)$, $(x',x_3)\in\R^2\times\R$ and we remark that
$$\begin{aligned}\norm{\mathcal F(q_\rho)-\mathcal F(q)}_{L^2(\R^3)}^2=\norm{q_\rho-q}_{L^2(\R^3)}^2&\leq \int_{\R^3}(1-\psi^2(\rho^{-\frac{1}{4}} x_3))|q(x)|^2dx\\
\ &\leq \int_{|x_3|\geq \rho^{\frac{1}{4}}}\left(\int_{\R^2}|q(x',x_3)|^2dx'\right)dx_3.\end{aligned}$$
Combining this with the fact that, according to Fubini's theorem,
$$x_3\mapsto \left(\int_{\R^2}|q(x',x_3)|^2dx'\right)\in L^1(\R),$$
we deduce that
$$\lim_{\rho\to+\infty}\norm{\mathcal F(q_\rho)-\mathcal F(q)}_{L^2(\R^3)}=0.$$
Thus, there exists a sequence $(\rho_k)_{k\in\mathbb N}$ such that $\rho_k\to+\infty$ and for a.e. $\xi\in\R^3$ we have
$$\lim_{k\to+\infty}\mathcal F(q_{\rho_k})(\xi)=\mathcal F(q)(\xi).$$
Combining this with \eqref{t1g}, we obtain that $\mathcal F(q)=0$ which implies that $q=0$ and $q_1=q_2$. This completes the proof of Theorem \ref{t1}.

\section{Recovery from measurements on a bounded portion of $\pd\Omega$}

In this section we will prove  Theorem \ref{c1} and we assume  that  the conditions of this theorem  are fulfilled. Recall that $\tau_0$ denotes the extension of the map $u\mapsto u_{|\pd\Omega}$ to $u\in H^1(\Omega)$ which takes values in $H^{\frac{1}{2}}_{loc}(\pd\Omega)$. Consider the sets of functions
$$Q_{A,q}:=\{u\in H^1(\Omega):\ \Delta_{A} u+qu=0\},$$
$$ Q_{A,q,r}:=\{ u\in Q_{A,q}: \ \textrm{supp}(\tau_0 u)\subset S_r\},\quad j=1,2.$$
Here we recall that $S_r=\pd\Omega\cap(\overline{\omega}\times[-r,r])$. We have the following density result.

\begin{proposition}\label{l5}  The space $Q_{A_1,q_1,r}$ $($resp. $Q_{A_2,\overline{q_2},r}$$)$  is dense in $Q_{A_1,q_1}$ $($resp. $Q_{A_2,\overline{q_2}}$$)$ for the topology induced by $L^2(\Omega\setminus(\Omega_{-}\cup\Omega_{+}))$.\end{proposition}
\begin{proof} The proof of these two results being similar, we will only show the density of  $Q_{A_1,q_1,r}$ in $Q_{A_1,q_1}$. We will prove the proposition by contradiction. Assume that   $Q_{A_1,q_1,r}$  is not dense in $Q_{A_1,q_1}$.
Then, there exist $h\in L^2(\Omega\setminus(\Omega_{-}\cup\Omega_{+}))$ and $v_0\in Q_{A_1,q_1}$ such that
\bel{l5a}\int_{\Omega\setminus(\Omega_{-}\cup\Omega_{+})} h\overline{v}dx=0,\quad v\in Q_{A_1,q_1,r},\ee
\bel{l5b}\int_{\Omega\setminus(\Omega_{-}\cup\Omega_{+})} h\overline{v_0}dx\neq0.\ee
Let us  mention that in contrast to several other  related density result (e.g. \cite[Proposition 3.1]{KLU} and \cite[Lemma 6.1]{Ki4}) we consider a general unbounded Lipschitz domain and we can not apply the Green formula in the usual sense. To avoid such difficulties, here we  proceed differently than other related results.

From now on, we extend $h$ by $0$ to $\Omega$. In view of Assumption 1, there exists $u\in H^1_0(\Omega)$ such that $\Delta_{A_1}u + \overline{q_1}u=h$. 
  Then, condition \eqref{l5a} implies
\bel{l5c}\int_{\Omega} (\Delta_{A_1} + \overline{q_1} ) u\overline{v}dx=0,\quad v\in Q_{A_1,q_1,r}.\ee
Moreover, for any $\phi\in\mathcal C^\infty_0(\Omega)$ and any $w\in H^1(\Omega)$,  we have
\bel{ll5aa}\begin{aligned} 2i\int_\Omega (A_1\cdot\nabla\phi)\overline{w}dx &=2i\overline{\left\langle wA_1, \nabla\phi \right\rangle_{\left(\mathcal C^\infty_0(\Omega)^3\right)',\mathcal C^\infty_0(\Omega)^3}}\\
\ &=-2i\overline{\left\langle \textrm{div}(wA_1), \phi \right\rangle_{D'(\Omega),\mathcal C^\infty_0(\Omega)}}\\
\ &=-2i\int_\Omega \textrm{div}(A_1)\phi\overline{w} dx+\int_\Omega \phi (\overline{2iA_1\cdot\nabla w}) dx.\end{aligned}\ee
By density we can extend this identity to $\phi\in H^1_0(\Omega)$. Combining this with the fact that $u\in H^1_0(\Omega)$,  for any $v\in Q_{A_1,q_1,r}$, we obtain
\bel{ll5a}\begin{aligned}\int_\Omega \Delta u\overline{v}dx-\int_\Omega  u\overline{\Delta v}dx&=\int_{\Omega} (\Delta_{A_1} + \overline{q_1}) u\overline{v}dx-\int_{\Omega} u \overline{(\Delta_{A_1} + q_1 )v}dx\\
\ &=\int_{\Omega\setminus(\Omega_{-}\cup\Omega_{+})} h\overline{v}dx=0.\end{aligned}\ee
On the other hand, in view of Assumption 1, for any $F\in\mathcal C^\infty_0(\R^3)$, satisfying supp$(F_{|\pd\Omega})\subset S_r$, we can define $w_{F}\in H^1_0(\Omega)$ solving $\Delta_{A_1}w_F+q_1w_F=-\Delta_{A_1}F+q_1F$ and $v=w_F+F\in Q_{A_1,q_1,r}$. Using this choice for the element $v\in Q_{A_1,q_1,r}$ in \eqref{ll5a}, we deduce that
\bel{ll5b}\int_\Omega \Delta u(\overline{w_F+F})dx-\int_\Omega  u(\overline{\Delta w_F+\Delta F})dx=0.\ee
In addition, since $u\in H^1_0(\Omega)$ and $w_F\in H^1_0(\Omega)$, one can check by density that
$$\int_\Omega \Delta u\overline{w_F}dx-\int_\Omega  u\overline{\Delta w_F}dx=-\int_\Omega \nabla u\cdot\overline{\nabla w_F}dx+\int_\Omega  \nabla u\cdot\overline{\nabla w_F}dx=0.$$
Combining this with \eqref{ll5b}, we get
\bel{ll5c}\int_\Omega \Delta u\overline{F}dx-\int_\Omega  u\overline{\Delta F}dx=0,\quad F\in\{G\in\mathcal C^\infty_0(\R^3): \textrm{supp}(G_{|\pd\Omega})\subset S_r\}.\ee
We fix $\gamma_1$ an open set of $\pd\Omega$ such that $\gamma_1\subset (S_r\setminus[\pd\Omega\cap (\overline{\omega}\times[\delta-r,r-\delta])])$. Then, we consider $\Omega_*$ a bounded subset of $\R^3\setminus\Omega$ with no empty interior such that $\Omega_*\cap\partial\Omega\subset\gamma_1$  and such that $\Omega_{-,*}:=\Omega_{-}\cup\Omega_*$ is an open connected set of $\R^3$. Applying \eqref{ll5aa} and \eqref{ll5c},  we deduce  that the extension of $u$ by zero to $\Omega_{-,*}$ satisfies
$$
\left\{
\begin{array}{l} 
(\Delta_{A_1} + \overline{q_1} )u  = 0\ \  \mbox{in}\ \Omega_{-,*},\\ 
u\in H^1(\Omega_{-,*})\\
u_{|\Omega_*}  = 0.
\end{array}
\right.
$$
Then, applying the unique continuation property for elliptic equations (e.g. \cite[Theorem 1.1]{GL} and \cite[Theorem 1]{SS}), we deduce that $u_{|\Omega_{-}}=0$. In the same way, we can prove that $u_{|\Omega_{+}}=0$. Using these properties, we would like to prove the following identity
\bel{ll5d}\int_{\Omega}\Delta_{A_1}u\overline{v_0}dx=\int_{\Omega}u\overline{\Delta_{A_1}v_0}dx,\ee
where we recall that $v_0$ satisfies \eqref{l5b}. For this purpose, we first recall that in a similar way to \eqref{ll5a}, we can show that
$$\int_\Omega \Delta u\overline{v_0}dx-\int_\Omega  u\overline{\Delta v_0}dx=\int_{\Omega} \Delta_{A_1}  u\overline{v_0}dx-\int_{\Omega} u \overline{\Delta_{A_1} v_0}dx.$$
Thus, we only need to prove that 
\bel{ll5e}\int_{\Omega}\Delta u\overline{v_0}dx=\int_{\Omega}u\overline{\Delta v_0}dx,\ee
for showing \eqref{ll5d}. Let $\phi_1,\phi_2\in\mathcal C^\infty_0(\R^3)$ be such that $\phi_1=1$ on  $\overline{\omega}\times\left[\frac{\delta}{2}-r,r-\frac{\delta}{2}\right]$, $\phi_2=1$ on a neighborhood of  supp$(\phi_1)$ and supp$(\phi_2)\cap\pd\Omega\subset (\overline{\omega}\times\left[\frac{\delta}{3}-r,r-\frac{\delta}{3}\right])$. Since supp$(\phi_2 v_0)\cap\pd\Omega\subset S_r$ and  $$\Delta_{A_1}(\phi_2 v_0)=-q_1\phi_2v_0 + 2\nabla\phi_2\cdot\nabla v_0+(\Delta_{A_1}\phi_2)v_0\in L^2(\Omega),$$ in a similar way to \eqref{ll5c}, we can apply Assumption 1 and \eqref{l5a} in order to get
\bel{ll5f}\int_\Omega \Delta u\overline{\phi_2 v_0}dx-\int_\Omega  u\overline{\Delta (\phi_2v_0)}dx=0.\ee
In addition, using the fact that $\phi_2=1$  on a neighborhood of supp$(\phi_1)$, we get
\bel{ll5g}\int_{\Omega}\Delta u(\overline{(1-\phi_2)v_0})dx=\int_{\Omega}\Delta [(1-\phi_1)u](\overline{(1-\phi_2)v_0})dx.\ee
On the other hand, using the fact that $$\Omega_-\cup \left(\omega\times\left[\frac{\delta}{2}-r,r-\frac{\delta}{2}\right]\cap\Omega\right)\cup\Omega_+$$
 corresponds to the intersection between a neighborhood of $\pd\Omega$ and  $\Omega$, with the fact that 
\bel{ll5gg}(1-\phi_1)u(x)=0,\quad x\in \Omega_-\cup \left(\omega\times\left[\frac{\delta}{2}-r,r-\frac{\delta}{2}\right]\cap\Omega\right)\cup\Omega_+,\ee
we deduce that the function $(1-\phi_1)u$ extended by zero to $\R^3$, satisfies $\nabla[(1-\phi_1)u]\in L^2(\R^3)$ and div$(\nabla[(1-\phi_1)u])=\Delta[(1-\phi_1)u]\in L^2(\R^3)$. Moreover, combining \eqref{ll5gg} with the arguments used in the proof of \cite[Theorem 3.4 page 223]{EE}, we can find a sequence of functions $(G_k)_{k\in\mathbb N}$ lying in $\mathcal C^\infty_0(\Omega)^3$ such that
$$\lim_{k\to+\infty}\norm{G_k-\nabla[(1-\phi_1)u]}_{L^2(\Omega)}=\lim_{k\to+\infty}\norm{\textrm{div}(G_k)-\Delta[(1-\phi_1)u]}_{L^2(\Omega)}=0.$$
Then, we have 
$$\begin{aligned}\int_{\Omega}\textrm{div}(G_k) (\overline{(1-\phi_2)v_0})dx&=\overline{\left\langle (1-\phi_2)v_0,\textrm{div}(G_k)\right\rangle_{D'(\Omega),\mathcal C^\infty_0(\Omega)}}\\
\ &=-\overline{\left\langle \nabla [(1-\phi_2)v_0],G_k\right\rangle_{\left(\mathcal C^\infty_0(\Omega)^3\right)',\mathcal C^\infty_0(\Omega)^3}}\\
\ &=-\int_{\Omega}G_k \cdot\overline{\nabla[(1-\phi_2)v_0]})dx\end{aligned}$$
and sending $k\to+\infty$, we obtain
$$\int_{\Omega}\Delta [(1-\phi_1)u](\overline{(1-\phi_2)v_0})dx=-\int_{\Omega}\nabla [(1-\phi_1)u]\cdot(\overline{\nabla[(1-\phi_2)v_0]})dx.$$
Then, using the fact that $(1-\phi_1)u\in H^1_0(\Omega)$, we find
$$\int_{\Omega}\Delta [(1-\phi_1)u](\overline{(1-\phi_2)v_0})dx=-\int_{\Omega}\nabla [(1-\phi_1)u]\cdot(\overline{\nabla[(1-\phi_2)v_0]})dx=\int_{\Omega}[(1-\phi_1)u](\overline{\Delta[(1-\phi_2)v_0]})dx.$$
Combining this with \eqref{ll5g} and applying again the fact that $\phi_2=1$  on a neighborhood of  supp$(\phi_1)$, we find
$$\int_{\Omega}\Delta u(\overline{(1-\phi_2)v_0})dx=\int_{\Omega}[(1-\phi_1)u](\overline{\Delta[(1-\phi_2)v_0]})dx=\int_{\Omega}u(\overline{\Delta[(1-\phi_2)v_0]})dx.$$
From this identity and \eqref{ll5f}, we deduce \eqref{ll5e} and by the same way \eqref{ll5d}. Applying \eqref{ll5d}, we find
$$\int_{\Omega} h\overline{v_0}dx=\int_{\Omega} (\Delta_{A_1} + \overline{q_1} ) u\overline{v_0}dx=\int_{\Omega}  u\overline{(\Delta_{A_1} + q_1 )v_0}dx=0.$$
This contradicts \eqref{l5b}. We have  completed the proof of the proposition.\end{proof}

Applying this proposition, we will complete the proof of Theorem \ref{c1}. \\
\ \\
\textbf{Proof of the Theorem \ref{c1}.} Let $u_1\in Q_{A_1,q_1,r}$ and $u_2\in Q_{A_2,\overline{q_2},r}$. In a similar way to Section 4, we can prove that \eqref{c1c} implies
\bel{c1e}i\int_{\Omega}(A\cdot\nabla u_1)\overline{u_2}dx -i\int_{\Omega}u_1(\overline{A\cdot\nabla u_2})dx+\int_{\Omega}\tilde{q}u_1\overline{u_2}dx =0,\ee
with $A=A_1-A_2$ and $\tilde{q}=|A_2|^2-|A_1|^2+q_1-q_2$. On the other hand, according to \eqref{c1d}, we have
$$\int_{\Omega} u_1(A\cdot\overline{\nabla u_2})dx=-\int_{\Omega}(A\cdot\nabla u_1)\overline{ u_2}dx-\int_{\Omega}\textrm{div}(A) u_1\overline{ u_2}dx.$$
Combining this with \eqref{c1e}, we obtain
$$2i\int_{\Omega}(A\cdot\nabla u_1)\overline{u_2}dx +\int_{\Omega}[\tilde{q}+i\textrm{div}(A)]u_1\overline{u_2}dx =0.$$
Then, \eqref{c1b} implies
$$2i\int_{\Omega\setminus(\Omega_-\cup\Omega_+)}(A\cdot\nabla u_1)\overline{u_2}dx +\int_{\Omega\setminus(\Omega_-\cup\Omega_+)}[\tilde{q}+i\textrm{div}(A)]u_1\overline{u_2}dx =0.$$
Applying Lemma \ref{l5}, we deduce by density that this last identity holds true for any $u_1\in Q_{A_1,q_1,r}$ and any $u_2\in Q_{A_2,\overline{q_2}}$. Then applying again \eqref{c1d} and \eqref{c1b},  we deduce that \eqref{c1e} holds true for any $u_1\in Q_{A_1,q_1,r}$ and any $u_2\in Q_{A_2,\overline{q_2}}$. In the same way, applying \eqref{c1d} and \eqref{c1b}, we can prove that  \eqref{c1e} holds true for any $u_1\in Q_{A_1,q_1}$ and any $u_2\in Q_{A_2,\overline{q_2}}$. Finally, choosing $u_1,u_2$ in a similar way to Section 4, we can deduce that $dA_1=dA_2$. Then by repeating the arguments  at the end of Section 4, we deduce that, for $q_1-q_2\in L^2(\Omega)$, we have $q_1=q_2$.\qed

\section{The partial data result}

This section is devoted to the proof of Theorem \ref{t6}. For all $y\in\mathbb S^{1}$,  $r>0$, we set
\[\partial\omega_{+,r,y}=\{x\in\pd\omega:\ \nu(x)\cdot y>r\},\quad\partial\omega_{-,r,y}=\{x\in\pd\omega:\ \nu(x)\cdot y\leq r\}.\]
We assume that $\Omega=\omega\times\R$ and, without lost of generality, we assume that there exists $\epsilon>0$ such that for any $\theta\in\{y\in\mathbb S^{1}:|y-\theta_0|\leq\epsilon\} $ we have $\partial\omega_{-,\epsilon,\theta}\subset V'$. 
We consider   $\rho >\max(\rho_2,\rho_1') $, with $\rho_1'$ given in Corollary \ref{c2} and $\rho_2$ defined in Proposition \ref{p2}, and we fix $\theta\in\{y\in\mathbb S^{1}:|y-\theta_0|\leq\epsilon\} $, $\xi:=(\xi',\xi_3)\in\R^3$ satisfying $\xi_3\neq0$ and $\xi'\in\theta^{\bot}\setminus\{0\}$. 
Then,  we fix $u_1\in H^{1}(\Omega)$  a solution of $\Delta_{A_1} u_1+q_1u_1=0$ in $\Omega$ and $u_2\in H^{1}(\Omega)$  a solution of $\Delta_{A_2} u_2+\overline{q_2}u_2=0$ in $\Omega$ of the form \eqref{CGO1}-\eqref{CGO2} with $\rho>\rho_2$ and with $w_{j,\rho}$ satisfying \eqref{RCGO}.  Following the argumentation of Section 3, used for proving the decay property of $w_{j,\rho}$ which is given for $j=1$ by \eqref{tutu}, we can show that
$$\rho^{-1}\norm{w_{j,\rho}}_{H^1(\Omega)}+\norm{w_{j,\rho}}_{L^2(\Omega)}\leq C(\norm{A_{j}-A_{j,\rho}}_{L^2(\R^3)^3}+\rho^{-\frac{1}{8}})$$
and assuming that $\rho^{-\frac{1}{8}}$ admits a faster decay than $\norm{A_{j}-A_{j,\rho}}_{L^2(\R^3)^3}$ we get
\bel{t6f}\rho^{-1}\norm{w_{j,\rho}}_{H^1(\Omega)}+\norm{w_{j,\rho}}_{L^2(\Omega)}\leq C\norm{A_{j}-A_{j,\rho}}_{L^2(\R^3)^3}.\ee
In view of \eqref{t6a}, there exists $v_2\in H^1(\Omega)$ satisfying $\Delta_{A_2} v_2+q_2v_2=0$ and $\tau v_2=\tau u_1$, ${N_{A_2,q_2}v_2}_{|V}={N_{A_1,q_1}u_1}_{|V}$.  Combining this with \eqref{c1d} we deduce that 
$u=v_2-u_1$ solves the boundary value problem
 \bel{eq4}
\left\{\begin{array}{ll}
\Delta_{A_2} u+q_2u=2iA\cdot \nabla u_1+(q+i\textrm{div}(A)+|A_2|^2-|A_1|^2)u_1 &\mbox{in}\ \Omega,
\\
u=0 & \mathrm{on}\  \partial \Omega.\\
\end{array}\right.
\ee
In particular, we have
 $$\Delta u=-2iA_2\cdot \nabla u-(q_2+i\textrm{div}(A_2)-|A_2|^2)u+ 2iA\cdot \nabla u_1+(q+i\textrm{div}(A)+|A_2|^2-|A_1|^2)u_1\in L^2(\Omega)$$ and,  in view of \cite[Lemma 2.2]{CKS}, we deduce that $u\in H^2(\Omega)$.

Now let us show that $\pd_\nu u_{| V}=0$. We fix $w\in H^2(\Omega)$ satisfying supp$( w_{|\pd\Omega})\subset V$ and using the fact that ${N_{A_2,q_2}v_2}_{|V}={N_{A_1,q_1}u_1}_{|V}$, we get
$$\begin{aligned}0&=\left\langle N_{A_2,q_2}v_2,\tau w\right\rangle-\left\langle N_{A_1,q_1}u_1, \tau w\right\rangle\\
\ &=\int_\Omega (\nabla+iA_1)u_1\cdot \overline{(\nabla+iA_1)w}dx-\int_\Omega q_1u_1\overline{w}dx-\int_\Omega (\nabla+iA_2)v_2\cdot \overline{(\nabla+iA_2)w}dx+\int_\Omega q_2v_2\overline{w}dx\\
\ &=-\int_\Omega (\nabla+iA_2)u\cdot \overline{(\nabla+iA_2)w}dx+\int_\Omega q_2u\overline{w}dx+\int_\Omega [iu_1A\cdot\overline{\nabla w}-i(A\cdot\nabla u_1)\overline{w}-(|A_2|^2-|A_1|^2+q)u_1\overline{w}]dx.\end{aligned}$$
Applying  \eqref{c1d} and the fact that $u\in H^1_0(\Omega)$, we get
$$\begin{aligned}&\int_\Omega [iu_1A\cdot\overline{\nabla w}-i(A\cdot\nabla u_1)\overline{w}-(|A_2|^2-|A_1|^2+q)u_1\overline{w}]dx\\
&=-2i\int_\Omega (A\cdot\nabla u_1)\overline{w}dx-i\int_\Omega \textrm{div}(A) u_1\overline{w}dx-\int_\Omega (|A_2|^2-|A_1|^2+q)u_1\overline{w}dx\\
&=-\int_\Omega (\Delta_{A_2}u+q_2u)\overline{w}dx\\
&=-\int_\Omega \Delta u\overline{w}dx-2i\int_\Omega (A_2\cdot\nabla u)\overline{w}dx-i\int_\Omega \textrm{div}(A_2)u\overline{ w}dx+\int_\Omega (|A_2|^2-q_2)u\overline{w}dx\\
&=-\int_\Omega \Delta u\overline{w}dx-i\int_\Omega (A_2\cdot\nabla u)\overline{w}dx+i\int_\Omega A_2u\overline{ \nabla w}dx+\int_\Omega (|A_2|^2-q_2)u\overline{w}dx\\
&=-\int_\Omega \Delta u\overline{w}dx+\int_\Omega (\nabla+iA_2)u\cdot \overline{(\nabla+iA_2)w}dx-\int_\Omega \nabla u\cdot\overline{\nabla w}dx-\int_\Omega q_2u\overline{w}dx\end{aligned}$$
and it follows
$$\int_{\pd\Omega}\pd_\nu u\overline{w}d\sigma(x)=\int_\Omega \Delta u\overline{w}dx+\int_\Omega \nabla u\cdot\overline{\nabla w}dx=0.$$
Allowing $w\in H^2(\Omega)$, satisfying supp$( w_{|\pd\Omega})\subset V$, to be arbitrary, we deduce $\pd_\nu u_{| V}=0$. In the same way, multiplying \eqref{eq4} by $\overline{u_2}$ and then applying \eqref{c1d} and the Green formula, we get
\[\int_\Omega [2iA\cdot \nabla u_1\overline{u_2}+(q+i\textrm{div}(A)+|A_2|^2-|A_1|^2)u_1\overline{u_2}]dx=\int_{\pd\Omega} \partial_\nu u \overline{u_2}d\sigma(x).\]
Moreover, we have $\partial_\nu u_{| V}=0$ and  we get
\begin{equation}\label{t6e} \int_\Omega [2iA\cdot \nabla u_1\overline{u_2}+(q+i\textrm{div}(A)+|A_2|^2-|A_1|^2)u_1\overline{u_2}]dx=\int_{\pd\Omega\setminus V}\partial_\nu u\overline{u_2}d\sigma(x).\end{equation}
In view of \eqref{t6f},  we have
\bel{tatita}\norm{w_{2,\rho}}_{L^2(\pd \Omega)}\leq C\norm{w_{2,\rho}}_{H^{1}(\Omega)}^{\frac{1}{2}}\norm{w_{2,\rho}}_{L^2(\Omega)}^{\frac{1}{2}}\leq C \rho^{\frac{1}{2}}\norm{A_{2}-A_{2,\rho}}_{L^2(\R^3)^3}.\ee
Here we use the estimate
$$\norm{f}_{L^2(\pd\Omega)}\leq C\norm{f}_{H^{1}(\Omega)}^{\frac{1}{2}}\norm{f}_{L^2(\Omega)}^{\frac{1}{2}},\quad f\in H^1(\Omega),$$
which can be proved, in a similar way to bounded domains, by using local coordinates associated with $\pd\omega$ in order to transform, locally with respect to $x'\in\overline{\omega}$ for $x=(x',x_3)\in\overline{\omega}\times\R=\overline{\Omega}$, $\overline{\Omega}$ into the half space.
Applying \eqref{tatita} and  the Cauchy-Schwarz inequality, 
we obtain
\[\begin{aligned}\abs{\int_{\pd\Omega\setminus V}\partial_\nu u\overline{u_2}d\sigma(x)}&\leq\int_\R\int_{{\partial\omega}_{+,\epsilon,\theta}}\abs{\partial_\nu ue^{-\rho x'\cdot \theta}\left(\psi\left(\rho^{-\frac{1}{4}}x_3\right)b_{2,\rho}e^{i\rho x\cdot\eta}+w_{2,\rho}(x)\right)} d\sigma(x')dx_3 \\
 \ &\leq C\left(\int_{{\partial\omega}_{+,\epsilon,\theta}\times \R}\abs{e^{-\rho x'\cdot \theta}\partial_\nu u}^2d\sigma(x)\right)^{\frac{1}{2}}\left(\norm{\psi\left(\rho^{-\frac{1}{4}}\cdot\right)}_{L^2(\R)}+\norm{w_{2,\rho}}_{L^2(\pd \Omega)}\right)\\
\ &\leq C\rho^{\frac{1}{2}}\norm{A_{2}-A_{2,\rho}}_{L^2(\R^3)^3}\left(\int_{{\partial\omega}_{+,\epsilon,\theta}\times \R}\abs{e^{-\rho x'\cdot \theta}\partial_\nu u}^2d\sigma(x)\right)^{\frac{1}{2}}\end{aligned}\]
for some $C$ independent of $\rho$. 
This estimate and the Carleman estimate \eqref{c2a} implies
\begin{eqnarray}&&\abs{\int_\Omega [2iA\cdot \nabla u_1\overline{u_2}+(q+i\textrm{div}(A)+|A_2|^2-|A_1|^2)u_1\overline{u_2}dx}^2\cr
&&\leq C\rho\norm{A_{2}-A_{2,\rho}}_{L^2(\R^3)^3}^2\int_{{\partial\omega}_{+,\epsilon,\theta}\times \R}\abs{e^{-\rho x'\cdot \theta}\partial_\nu u}^2d\sigma(x)\cr
&&\leq \epsilon^{-1}C\rho\norm{A_{2}-A_{2,\rho}}_{L^2(\R^3)^3}^2\int_{{\partial\omega}_{+,\theta}\times \R}\abs{e^{-\rho x'\cdot \theta}\partial_\nu u}^2|\nu \cdot\theta| d\sigma(x)\cr
&&\leq \epsilon^{-1}C\norm{A_{2}-A_{2,\rho}}_{L^2(\R^3)^3}^2\left(\int_\Omega\abs{ e^{-\rho x'\cdot \theta}(-\Delta_{A_2} +q_2)u}^2dx\right)\cr
&&\leq \epsilon^{-1}C\norm{A_{2}-A_{2,\rho}}_{L^2(\R^3)^3}^2\left(\int_\Omega\abs{ e^{-\rho x'\cdot \theta}[2iA\cdot \nabla u_1+(q+i\textrm{div}(A)+|A_2|^2-|A_1|^2)u_1]}^2dx\right)\cr
&&\leq \epsilon^{-1}C\rho^2\norm{A_{2}-A_{2,\rho}}_{L^2(\R^3)^3}^2\norm{A}_{L^2(\R^3)}^2,\end{eqnarray}
where $C>0$ is a constant independent of $\rho$. Therefore, we have
$$\abs{\int_\Omega [2iA\cdot \nabla u_1\overline{u_2}+(q+i\textrm{div}(A)+|A_2|^2-|A_1|^2)u_1\overline{u_2}dx}\leq C\rho\norm{A_{2}-A_{2,\rho}}_{L^2(\R^3)^3}$$
and multiplying this inequality by $\rho^{-1}$ and sending $\rho\to+\infty$ we obtain from \eqref{a1a} that
$$\lim_{\rho\to+\infty}\rho^{-1}\abs{\int_\Omega [2iA\cdot \nabla u_1\overline{u_2}+(q+i\textrm{div}(A)+|A_2|^2-|A_1|^2)u_1\overline{u_2}dx}=0.$$
Combining this identity with the arguments of Section 4, we deduce that 
\bel{t6g}\xi_k \mathcal F(a_j)(\xi)-\xi_j \mathcal F(a_k)(\xi)=0,\quad 1\leq j<k\leq3\ee
for all $(\xi',\xi_3)\in\R^2\times\R$ such that $\xi'\in\theta^\bot\setminus\{0\}$, $\theta\in\{y\in\mathbb S^{1}:|y-\theta_0|\leq\epsilon\} $, $\xi_3\neq0$.
Since $A\in L^1(\R^3)$, we can extend by continuity the identity \eqref{t6g} to all $(\xi',\xi_3)\in\R^2\times\R$ such that $\xi'\in\theta^\bot$, $\theta\in\{y\in\mathbb S^{1}:|y-\theta_0|\leq\epsilon\} $, $\xi_3\in\R$. Consider the Fourier transform in $x'$ and $x_3$ given, for $f\in L^1(\R^3)$, by
$$\mathcal F'(f)(\xi',x_3)=(2\pi)^{-1}\int_{\R^2}f(x',x_3)e^{-ix'\cdot\xi'}dx',\quad \mathcal F_{x_3}(f)(x',\xi_3)=(2\pi)^{-\frac{1}{2}}\int_{\R}f(x',x_3)e^{-ix_3\xi_3'}dx_3.$$
It is clear that $\mathcal F A=\mathcal F'[\mathcal F_{x_3}A]$ and using the fact that, for all $\xi_3\in\R$, $x'\mapsto \mathcal F_{x_3}A(x',\xi_3)$ is  supported in $\overline{\omega}$ which is compact, we deduce that, for all $j=1,2,3$,  $\xi'\mapsto \mathcal Fa_j(\xi',\xi_3)$ is complex valued real analytic. Therefore, for all $\xi_3\in\R$, the function $\xi'\mapsto\xi_k \mathcal F(a_j)(\xi)-\xi_j \mathcal F(a_k)(\xi)$ is real analytic and it follows that the identity \eqref{t6g} holds true for all $\xi\in\R^3$.Thus, we have $dA_1=dA_2$. Then in a similar way to Section 4, we can prove that we can apply the gauge invariance to get
$$ \mathcal D_{A_1,q_1, V}=\mathcal D_{A_1,q_2,V}.$$
Repeating the above argumentation (see also \cite[Section 5]{Ki4}) we deduce that
$$\lim_{\rho\to+\infty}\int_{\R^3}\chi^2(\rho^{-\frac{1}{4}} x_3)q(x)e^{-i\xi\cdot x}dx=0,$$
for all $(\xi',\xi_3)\in\R^2\times\R$ such that $\xi'\in\theta^\bot\setminus\{0\}$, $\theta\in\{y\in\mathbb S^{1}:|y-\theta_0|\leq\epsilon\} $, $\xi_3\neq0$. Then, using the fact that $q\in L^1(\R^3)$, an application of the Lebesgue dominate convergence theorem implies that $\mathcal F(q)(\xi)=0$, 
for all $(\xi',\xi_3)\in\R^2\times\R$ such that $\xi'\in\theta^\bot$, $\theta\in\{y\in\mathbb S^{1}:|y-\theta_0|\leq\epsilon\} $, $\xi_3\in\R$. Then, using the fact that $q\in L^1(\R^3)$ and supp$(q)\subset\overline{\omega}\times\R$, we can repeat the above arguments in order to deduce that $q=0$ and $q_1=q_2$. This completes the proof of Theorem \ref{t6}.

\section{Extension to higher dimension}

In this section we discuss about  some possible extensions of our results to some class of domain $\Omega\subset\R^n$, $n\geq4$. For this purpose, let $n\geq4$ and consider $n_1,n_2\in\mathbb N$ such that $n_1+n_2=n$ and $n_1\geq3$. We fix also $\omega$ a bounded and $\mathcal C^2$  open set  of $\R^{n_1}$. Then our claim can be stated as follows: all the results of the present paper can be extended to any open and unbounded set $\Omega$ of $\R^n$ satisfying
\bel{op}\Omega\subset \Omega_2:=\omega\times\R^{n_2}.\ee
 Let us explain why our results can  also be extended to unbounded domains $\Omega$ satisfying \eqref{op}. The main ingredient are suitable CGO solutions for our problem. Once this is proved one can easily complete the proof of the uniqueness result by repeating our argumentation. Since here we know that $\omega$ is a bounded open set of $\R^{n_1}$ with $n_1\geq3$, instead of the construction of the present paper we will consider CGO solutions constructed by mean of a projection argument inspired by the analysis of \cite{BKS1,Ki1}. More precisely, we fix $\xi=(\xi',\xi'')\in\R^{n_1}\times\R^{n_2}$ and we consider $\eta,\theta\in\mathbb S^{n_1-1}$ such that $\eta\cdot\theta=\eta\cdot\xi'=\theta\cdot\xi'=0$.   For all $r>0$, we denote by $B_r'$ the ball of center zero and of radius $r$ of $\R^{n_1}$, we fix also $R:=\underset{x'\in\overline{\omega}}{\sup}|x'|$, $R_1:=2\sqrt{2}(R+2)$, $\tilde{\theta}=(\theta,0)\in\R^n$ and  $\tilde{\eta}=(\eta,0)\in\R^n$.  We set $\chi\in\mathcal C^\infty_0(\R^{n})$ such that $\chi\geq0$, $\int_{\R^{n}}\chi(x)dx=1$, supp$(\chi)\subset \{x\in\R^n:\ |x|<1\}$, and we define $\chi_\rho$ by
$\chi_\rho(x)=\rho^{{n\over 4}}\chi(\rho^{{1\over4}}x)$. Then, for $j=1,2$, we fix
$$A_{j,\rho}(x):=\int_{\R^{n}}\chi_\rho(x-y)A_j(y)dy.$$
In a similar way to Section 3.1, one can check that for all $x=(x',x'')\in B_{R+1}'\times\R^{n_2}$ 
the function
$$(s_1,s_2)\mapsto A_{j,\rho}(s_1\tilde{\theta}+s_2\tilde{\eta}+x)$$
will be supported in $\{z\in\R^2:\ |z|<R_1\}$. Thus,  we can define
$$ \begin{aligned}&\Phi_{1,\rho}(x):=\frac{-i}{2\pi} \int_{\R^2} \frac{(\tilde{\theta}+i\tilde{\eta})\cdot A_{1,\rho}(x-s_1\tilde{\theta}-s_2\tilde{\eta})}{s_1+is_2}ds_1ds_2,\\
 &\Phi_{2,\rho}(x):=\frac{-i}{2\pi} \int_{\R^2} \frac{(-\tilde{\theta}+i\tilde{\eta})\cdot A_{2,\rho}(x+s_1\tilde{\theta}-s_2\tilde{\eta})}{s_1+is_2}ds_1ds_2.\end{aligned}$$
Fixing
$$ b_{1,\rho}(x)=e^{\Phi_{1,\rho}(x)},\quad b_{2,\rho}(x)=e^{\Phi_{2,\rho}(x)},$$
we will obtain functions satisfying properties similar to those described in Section 3.1. Now let us fix $\psi\in\mathcal C^\infty_0(\R^{n_2})$  a real valued function. Applying the results of Section 3.2, which can be extended without any difficulty to this setting, one can construct solutions $u_j\in H^1(\Omega_2)$, $j=1,2$, of $\Delta_{A_j}u_j+q_ju_j=0$ on $\Omega_2$ of the form 
$$u_1(x',x'')=e^{\rho \theta\cdot x'}\left(\psi(x'')b_{1,\rho}(x',x'')e^{i\rho x'\cdot\eta-i\xi\cdot x}+w_{1,\rho}(x',x'')\right),\quad x'\in\omega,\ x''\in\R^{n_2},$$
$$u_2(x',x'')=e^{-\rho \theta\cdot x'}\left(\psi(x'')b_{2,\rho}(x',x'')e^{i\rho x'\cdot\eta}+w_{2,\rho}(x',x'')\right),\quad x'\in\omega,\ x''\in\R^{n_2},$$
with $w_j$ satisfying the decay property
$$\lim_{\rho\to+\infty}(\rho^{-1}\norm{w_{j,\rho}}_{H^1(\Omega_2)}+\norm{w_{j,\rho}}_{L^2(\Omega_2)})=0.$$
After that, allowing the cut-off function $\psi\in\mathcal C^\infty_0(\R^{n_2})$ to be arbitrary and repeating the arguments of Section 4 we can prove that all the results of this paper remain true when $\Omega\subset\R^n$ satisfies \eqref{op}.

\section*{Acknowledgments}

The  author would like to thank Pedro Caro for   fruitful discussions about recovery of bounded magnetic potentials. The
author is grateful to the anonymous referees for their careful reading and their suggestions that allow to
improve the paper. This work was partially supported by  the French National
Research Agency ANR (project MultiOnde) grant ANR-17-CE40-0029.

\end{document}